\documentclass[11pt,reqno]{amsart}
\usepackage{amsmath,amssymb,amsthm,mathrsfs,ifpdf}
\usepackage[hidelinks,hypertexnames=false]{hyperref}
\makeatletter
\renewcommand{\@biblabel}[1]{[#1]}
\makeatother
\ifpdf
  \usepackage{embedfile}
\fi
\numberwithin{equation}{section}
\newtheorem{theorem}{Theorem}[section]
\newtheorem{lemma}[theorem]{Lemma}
\newtheorem{proposition}[theorem]{Proposition}
\newtheorem{corollary}[theorem]{Corollary}
\theoremstyle{definition}

\theoremstyle{remark}
\newtheorem{remark}[theorem]{Remark}
\newcommand{\R}{\mathbb R}
\newcommand{\Z}{\mathbb Z}
\newcommand{\tr}{\operatorname{tr}}
\newcommand{\covol}{\operatorname{covol}}
\newcommand{\Sym}{\operatorname{Sym}}
\newcommand{\GL}{\operatorname{GL}}
\newcommand{\D}{\mathcal D_4}
\newcommand{\Ry}{\mathscr R}
\newcommand{\DD}{\mathscr D}
\newcommand{\Acl}{\mathscr A}
\newcommand{\T}{\mathscr T}
\newcommand{\dd}{\,\mathrm d}
\newcommand{\sodd}{\sigma_{\rm odd}}
\newcommand{\Hh}{\mathcal H}
\newcommand{\eps}{\varepsilon}
\DeclareMathOperator*{\argmin}{arg\,min}

\title[Theta minimization in dimension four]
{On minima of theta and Epstein zeta functions in dimension four}
\author{Senping Luo}
\address[S.~Luo]{School of Mathematics and Statistics, Jiangxi Normal
University, Nanchang 330022, China}
\email[S.~Luo]{luosp1989@163.com}
\author{Juncheng Wei}
\address[J.~Wei]{Department of Mathematics, Chinese University of
Hong Kong, Shatin, NT, Hong Kong}
\email[J.~Wei]{wei@math.cuhk.edu.hk}
\date{}
\begin{document}
\begin{abstract}
Let the theta and Epstein zeta functions be
$\Theta(\alpha,L)=\sum_{v\in L}e^{-\pi\alpha|v|^2}$ for $\alpha>0$
and
$E(L,s)=\sum_{v\in L\setminus\{0\}}{|v|^{-2s}}$
for $s>2$, respectively. We consider full-rank lattices $L\subset\R^4$.
Let the covolume of $L$ be one, and let
$\mathcal D_4=2^{-1/4}D_4$ be the root lattice $D_4$ rescaled to covolume one.
We prove that, up to orthogonal transformations,
\begin{equation}\nonumber
\argmin_{\covol(L)=1}\Theta(\alpha,L)=\mathcal D_4
\qquad\text{if}\quad \alpha>0,
\end{equation}
and this implies that
\begin{equation}\nonumber
\argmin_{\covol(L)=1}E(L,s)=\mathcal D_4
\qquad\text{if}\quad s>2.
\end{equation}
This proves that the root lattice minimizes the theta and Epstein zeta functions among all lattices in dimension four,
as conjectured by Sarnak-Str\"ombergsson (\cite{SS}, 2006), who established the local minimality of $\mathcal D_4$.
Thereby, this resolves the Rankin-Sobolev problem in dimension four dating back to \'Endibaev (\cite{End1978}, 1978) and Shushbaev (\cite{Shu1978}, 1978).
\end{abstract}
\maketitle
\raggedbottom

\section{Introduction}\label{sec:intro}

\subsection{Main results and background}
For a full-rank lattice $L\subset\mathbb R^d$, define the theta and
Epstein zeta functions by
\begin{equation}\label{eq:theta}
 \Theta(\alpha,L)=\sum_{v\in L}e^{-\pi\alpha|v|^2},\quad \alpha>0,
\end{equation}
and
\begin{equation}\label{eq:Epstein-definition}
 E(L,s)=\sum_{v\in L\setminus\{0\}}\frac1{|v|^{2s}},
 \quad s>\frac{d}{2},
\end{equation}
where $|\cdot|$ denotes the Euclidean norm. Unless stated otherwise,
lattices have covolume one and are identified up to orthogonal
transformations. Let
\begin{equation}\label{eq:D}
 D_4=\{m\in\Z^4:m_1+m_2+m_3+m_4\in2\Z\},
 \qquad \D=2^{-1/4}D_4.
\end{equation}
The lattice $\D$ has covolume one and minimum squared length $\sqrt2$.

Locating the minima of $E(L,s)$ or $\Theta(\alpha,L)$ has a long
history. The problem is often called the Rankin--Sobolev
problem~\cite{Rankin,Sobo1961}, particularly in the Russian
literature; see, for example,~\cite{Shu1978}. For two fundamental,
groundbreaking, and celebrated contributions in dimensions three,
eight, and twenty-four, we refer to
Sarnak--Str\"ombergsson~\cite{SS} and
Cohn--Kumar--Miller--Radchenko--Viazovska~\cite{CKMRV}.

Our first main result concerns the theta function in dimension four.
\begin{theorem}[Minima of 4d theta function]
\label{thm:main}
Let $L\subset\R^4$ be a lattice of covolume one. For every $\alpha>0$,
\begin{equation}\label{eq:main}
 \Theta(\alpha,L)\ge\Theta(\alpha,\D).
\end{equation}
Equality holds if and only if $L$ is orthogonally equivalent to the root lattice $\D$.
\end{theorem}

Our second main theorem concerns the Epstein zeta function in dimension four and follows
from Theorem~\ref{thm:main} via the Mellin transform:
\begin{theorem}[Minima of 4d Epstein zeta function]
\label{thm:Epstein}
Let $L\subset\R^4$ be a lattice of covolume one, and let $s>2$. Then
\begin{equation}\label{eq:Epstein-main}
 E(L,s)\ge E(\D,s).
\end{equation}
Equality holds if and only if $L$ is orthogonally equivalent to the root lattice $\D$.
\end{theorem}

Poisson summation and the Mellin transform give the meromorphic
continuation and functional equation of the Epstein zeta
function~\cite{Epstein}. After rescaling time, $\Theta(\alpha,L)$ is
the heat trace of the flat torus $\R^4/L^*$, where $L^*$ is the
Euclidean dual. The derivative of $(4\pi^2)^{-s}E(L,s)$ at zero
determines the regularized Laplacian determinant and spectral
height~\cite{SS}.

In dimension four, $D_4$ and $A_4$ are the only perfect forms up to
integral equivalence and scaling~\cite{KZ,BC,Barnes,Sch}.
Sarnak--Str\"ombergsson~\cite{SS} proved strict local minimality of
$\D$ for theta at every positive scale and for the analytically
continued Epstein zeta function at every positive real parameter
other than its pole. Spherical $4$-designs in every nonzero shell
give strict local zeta minimality in the convergence range
(Coulangeon~\cite{Coul}). Global Gaussian optimality of $D_4$ has
numerical support (Cohn--Kumar--Sch\"urmann~\cite{CKS}); earlier
dimension reductions concern specified layered, translated, and
orthorhombic families
(B\'etermin--Petrache~\cite{BeterminPetrache2017}).

In our proof, Theorem~\ref{thm:main} implies
Theorem~\ref{thm:Epstein} through the positive-superposition
principle for lattice energies; see the work of
B\'etermin~\cite{Betermin2021Defects}, inspired by
Cohn--Kumar~\cite{CK}. Historically, the Epstein zeta minimization
problem has attracted more attention than its theta counterpart.
In dimension two, Rankin~\cite{Rankin} initiated the study of
Epstein zeta minima for $s\ge1.035$ and proved that the hexagonal
lattice is optimal. Cassels~\cite{Cassels}, Ennola~\cite{Ennola},
and Diananda~\cite{Diananda} subsequently extended the parameter
range and developed further methods. Montgomery~\cite{Mont}
proved that the hexagonal lattice minimizes theta for every
$\alpha>0$, thereby implying Rankin's theorem. This is now a
celebrated result in lattice minimization; see
B\'etermin~\cite{Betermin}.

For interesting progress in dimension three, we refer to
B\'etermin--\v{S}amaj--Trav\v{e}nec~\cite{BeterminSAM2023},
B\'etermin--Petrache~\cite{BeterminPetrache2017}, and
Sarnak--Str\"ombergsson~\cite{SS}, together with the references
therein. Our previous work~\cite{LuoWei2026} proves that the
face-centred cubic lattice minimizes Epstein zeta, and that the
face-centred cubic lattice minimizes theta for $\alpha>1$, while
the body-centred cubic lattice minimizes theta for $\alpha<1$,
as conjectured by Sarnak--Str\"ombergsson~\cite{SS}.
In dimension four, relatively few results are available. Early
claims of local minimality appear in
\'Endibaev~\cite{End1978} and Shushbaev~\cite{Shu1978,Shu1989};
rigorous local minimality of $\D$ for theta and Epstein zeta was
established by Sarnak--Str\"ombergsson~\cite{SS}. Partial results
and further references to earlier work can be found in
Lim--Teo~\cite{Lim}.

In dimensions eight and twenty-four,
Sarnak--Str\"ombergsson~\cite{SS} proved local minimality of $E_8$
and the Leech lattice for theta and Epstein zeta.
Cohn--Kumar--Miller--Radchenko--Viazovska~\cite{CKMRV} proved the
deep universal optimality theorem for their interaction energies,
which in particular establishes global theta and zeta minimality.
Fourier linear programming connects sphere packing (Cohn-Elkies~\cite{CE}) with
universal optimality for completely monotone functions of squared
distance (Cohn-Kumar~\cite{CK}). This approach proves universal optimality of
$E_8$ and the Leech lattice among fixed-density point
configurations~\cite{CKMRV}, and hence minimality of Coulomb and
Riesz renormalized energies (Petrache-Serfaty~\cite{PetracheSerfaty2020}).

These developments arise mainly from analytic number theory and
mathematical analysis. Our interest comes from the $f$-potential
lattice energy
\begin{equation}\label{EFL}
 E_f(L)=\sum_{v\in L\setminus\{0\}}f(|v|^2).
\end{equation}
Gaussian and Riesz potentials give $\Theta(\alpha,L)-1$ and
$E(L,s)$, respectively. Such energies model ground states of
two-body lattice systems and phase transitions in crystals arising
in statistical physics and chemistry. B\'etermin and his
collaborators~\cite{Betermin,Betermin2018Local,BeterminCubic2019,
BeterminMorse2019,Betermin2021Defects,BeterminBonds2021,
BeterminDeLucaPetrache2021,Betermin2021LMP,BeterminPetrache2017,
BeterminPetrache2019,BeterminSAM2023,Betermin2024AMP,BeterminZhang2015}
initiated and extensively developed the minimization theory of
lattice energies in dimensions two and three, with particular
attention to physically relevant Riesz, Gaussian, Lennard--Jones,
and Morse potentials. Inspired by two-dimensional lattice energies,
they also introduced and studied new types of theta
functions~\cite{Betermin2020Soft,BeterminFaulhuber2023,BeterminFS2021,
BeterminSandier2018}. These generalized theta functions extend
naturally to higher dimensions and deserve further study.

In dimension two, we have systematically investigated
two-component Coulomb energies in diblock
copolymers~\cite{LuoRenWei2020}, Mueller--Ho energies in Bose--Einstein
condensates~\cite{LuoWeiThetaSums}, energies associated with Gaussian
potentials~\cite{LuoWeiThetaDifferences,LuoWeiNonmonotone,
LuoWeiNonGaussian,LuoWeiThetaDerivative,LuoWeiRatios}, and
Lennard--Jones interactions~\cite{LW}. Two related important works
apply modular functions to copolymer and Abrikosov vortex problems:
Chen--Oshita~\cite{Che2007} and
Sandier--Serfaty~\cite{Serfaty2012}, respectively. In both settings,
the hexagonal lattice is optimal.

This direction has attracted sustained attention; see the reviews
by Radin~\cite{Radin1987,Radin1991},
Ca\~nizo--Ramos-Lora~\cite{Canizo2024}, and
Lewin~\cite{Lewin2022,Lewin2025}. It is closely connected to the
crystallization conjecture, for which we refer to the comprehensive
review by Blanc--Lewin~\cite{Lewin2015}. For crystallization under
certain nonmonotone potentials, we refer to the celebrated results
of Theil~\cite{Theil2006} in dimension two and
Flatley--Theil~\cite{Theil2015} in dimension three.

\subsection{Lattices and reduction of positive definite forms}
\label{subsec:reduction-reference}
For a lattice $L$, write
\[
 a(L)=\min_{v\in L\setminus\{0\}}|v|^2
\]
for its minimum squared length. Thus $\sqrt{a(L)}$ is the distance
between nearest lattice points, and $\sqrt{a(L)}/2$ is the packing
radius. This quantity is independent of the choice of basis and
satisfies $a(cL)=c^2a(L)$ for $c>0$.

We use Gram matrices to represent lattice shapes and Minkowski
reduction to obtain compact representatives. Let $\Sym_4^+$ denote
the cone of real symmetric positive definite $4\times4$ matrices.
For a basis matrix $B=(b_1,\ldots,b_4)$ of $L=B\Z^4$, write
$Q=B^TB$, $Q[m]=m^TQm$ and
$a(Q)=\min_{m\in\Z^4\setminus\{0\}}Q[m]=a(L)$.
Then
\[
 \det Q=\covol(L)^2,\qquad
 \Theta(\alpha,Q)=\sum_{m\in\Z^4}e^{-\pi\alpha Q[m]}=\Theta(\alpha,L).
\]
A basis change replaces $Q$ by $U^TQU$, $U\in\GL_4(\Z)$.
The Euclidean dual is
\[
 L^*=\{y\in\R^4:y\cdot v\in\Z\text{ for all }v\in L\};
\]
the dual basis $B^{-T}$ has Gram matrix $Q^{-1}$.

Alongside the coordinate model for $D_4$ in \eqref{eq:D}, we use
\begin{equation}\label{eq:A-model}
 A_4=\left\{v\in\Z^5:\sum_{i=1}^5v_i=0\right\}
 \subset H,\qquad
 H=\left\{y\in\R^5:\sum_{i=1}^5y_i=0\right\}.
\end{equation}
The hyperplane $H$ has the induced Euclidean metric and is identified
isometrically with $\R^4$. Writing $e_i$ for the standard coordinate
vectors in the respective ambient spaces, choose the bases
\[
 B_D=(e_1-e_2,e_2-e_3,e_3-e_4,e_3+e_4),\qquad
 B_A=(e_1-e_5,e_2-e_5,e_3-e_5,e_4-e_5).
\]
Their Gram matrices are
\begin{equation}\label{eq:root-Gram}
 G_D=B_D^TB_D=
 \begin{pmatrix}
 2&-1&0&0\\-1&2&-1&-1\\0&-1&2&0\\0&-1&0&2
 \end{pmatrix},\qquad
 G_A=B_A^TB_A=I_4+\boldsymbol1\boldsymbol1^T,
\end{equation}
where $\boldsymbol1=(1,1,1,1)^T$. Both root lattices have minimum
squared length $2$, and their covolumes are $2$ and $\sqrt5$,
respectively. Thus the forms of minimum squared length one are
\begin{equation}\label{eq:minimum-one-forms}
 P_D=\tfrac12G_D,\qquad P_A=\tfrac12G_A,\qquad
 \det P_D=\tfrac14,\qquad \det P_A=\tfrac5{16}.
\end{equation}
We denote their integral-basis orbits by
\begin{equation}\label{eq:perfect-orbits}
 \DD=\{U^TP_DU:U\in\GL_4(\Z)\},\qquad
 \Acl=\{U^TP_AU:U\in\GL_4(\Z)\}.
\end{equation}
The covolume-one normalizations are $\D=2^{-1/4}D_4$ and
$5^{-1/8}A_4$; in particular, a Gram matrix of $\D$ is
$2^{-1/2}G_D=\sqrt2P_D$.

Pairing with the roots $e_i\pm e_j$ gives the dual lattice explicitly:
\begin{equation}\label{eq:D-dual-model}
 D_4^*=\Z^4\cup\left(\Z^4+\tfrac12\boldsymbol1\right)
       =B_D^{-T}\Z^4.
\end{equation}
It has Gram matrix $G_D^{-1}$, covolume $1/2$, and minimum squared
length $1$. Consequently $\D^*=2^{1/4}D_4^*$; the orthogonal
equivalence of $\D^*$ and $\D$ is given in Lemma~\ref{lem:isodual}.

The forms with minimum at least one constitute the Ryshkov domain
\begin{equation}\label{eq:Ry}
 \Ry=\{R\in\Sym_4^+:R[m]\ge1\text{ for all }m\in\Z^4\setminus\{0\}\}.
\end{equation}
For a covolume-one lattice with Gram matrix $Q$, the normalization
$R=Q/a(L)$ represents $a(L)^{-1/2}L$ and satisfies
\[
 a(R)=1,\qquad R\in\Ry,\qquad \det R=a(L)^{-4}.
\]
Thus lower bounds for the determinant on $\Ry$ give upper bounds
for the shortest squared length at fixed covolume. This is the
role of $a(L)$ in the geometric reduction of Section~\ref{sec:reduction}.

A form $P\in\Ry$ with $a(P)=1$ is \emph{perfect} if the matrices
$mm^T$, $m\in\Z^4$, with $P[m]=1$ span the space of real symmetric $4\times4$
matrices. The vertices of $\Ry$ are precisely the forms in
$\DD\cup\Acl$, with determinants $1/4$ and $5/16$,
respectively~\cite{KZ,Sch}. For $P\in\DD\cup\Acl$, set
\begin{equation}\label{eq:traceh}
 h_P(R)=\tr(P^{-1}R)-4.
\end{equation}
For fixed $P$ and $r>0$, the set $\{R\in\Ry:h_P(R)<r\}$ is the
corresponding trace neighborhood.
We take these neighborhoods over all integral-basis representatives;
they need not be disjoint.

A matrix $Q=(q_{ij})\in\Sym_4^+$ is \emph{Minkowski reduced} if
\begin{equation}\label{eq:Minkowski-definition}
 Q[m]\ge q_{ii}\quad\text{whenever }m\in\Z^4
 \text{ and }\gcd(m_i,\ldots,m_4)=1,\qquad i=1,\ldots,4.
\end{equation}
Equivalently, each $b_i$ is a shortest vector among those for which
$b_1,\ldots,b_{i-1},b_i$ can be extended to a lattice basis.
In dimension four, Barnes--Cohn~\cite{BC} give the following finite
description; see also Barnes~\cite[Section~4]{Barnes}:
\begin{equation}\label{eq:order}
 q_{11}\le q_{22}\le q_{33}\le q_{44},
\end{equation}
and, for $i=2,3,4$,
\begin{equation}\label{eq:36}
 Q[(m_1,\ldots,m_{i-1},1,0,\ldots,0)]\ge q_{ii},\qquad
 (m_1,\ldots,m_{i-1})\in\{-1,0,1\}^{i-1}\setminus\{0\}.
\end{equation}
Every positive definite form is integrally equivalent to a reduced
form, for which $q_{11}=a(Q)$ and $2|q_{ij}|\le q_{ii}$ for $i<j$.
The representatives $P_D$ and $P_A$ above are reduced. The fundamental
inequality~\cite[p.~46, equations (1.1)--(1.2)]{Barnes} gives
\begin{equation}\label{eq:Barnes}
 \det Q\ge\tfrac14q_{11}q_{22}q_{33}q_{44}.
\end{equation}
For covolume-one lattices, \eqref{eq:Barnes} gives
$0<a(L)\le\sqrt2=a(\D)$.
Reduction and \eqref{eq:Barnes} provide the compact representatives
used in Lemmas~\ref{lem:compact} and~\ref{lem:sublevels} and in the
attainment argument for Theorem~\ref{prop:barrier}. Theta is invariant
under integral changes of basis, as are the trace comparisons when
the same change is applied to $P$ and $R$.

\subsection{Outline of the proof}
\label{subsec:strategy}
The proof combines local and global comparisons. For a covolume-one
lattice $L$, set
\begin{equation}\label{eq:weighteddefect}
 \Phi(\alpha,L)=\alpha\bigl(\Theta(\alpha,L)-\Theta(\alpha,\D)\bigr).
\end{equation}
Poisson summation and the isoduality of $\D$ give
\[
 \Phi(\alpha,L)=\Phi(\alpha^{-1},L^*),
\]
so it suffices to consider $\alpha\ge1$.

In Section~\ref{sec:reduction}, tangent-cone estimates at the two
perfect forms and concavity of $\log\det$ give determinant
separation. In particular, if $\det Q=1$, $a=a(Q)$, and $R=Q/a$
satisfies
\[
 \inf_{P\in\DD}h_P(R)\ge\frac{11}{25},\qquad
 \inf_{P\in\Acl}h_P(R)\ge\frac1{10},
\]
then
\[
 a^{-4}=\det R\ge\frac{351^2}{25^4},\qquad
 a\le\frac{25}{\sqrt{351}}.
\]

In Section~\ref{sec:local}, shell moments give the local comparison
at $\D$, with equality only for its orthogonal images; root
averaging gives strict comparison near $A_4$.
A shortest antipodal pair gives
$\Theta(\alpha,L)-1\ge2e^{-\pi\alpha a(L)}$.
Hence an upper bound for $a(L)$ yields a lower bound for theta,
sufficient for comparison at large scales. A Fourier minorant
treats short vectors at the remaining scales. Outside these regions,
the remaining parameters satisfy
\[
 1\le\alpha<11,\qquad
 \frac{21}{20}<a(L)\le\frac{25}{\sqrt{351}}.
\]

In Section~\ref{sec:joint}, the extreme-scale comparison and compactness
of theta sublevels give a global minimum of $\Phi$ over all positive
scales and covolume-one shapes. Duality permits a minimizing scale
$\alpha\ge1$. Shape stationarity makes the Gaussian second-moment
matrix scalar, while scale stationarity relates its trace to the
theta difference. Retaining a shortest antipodal pair gives, with
$a=a(L)$,
\[
 \Theta(\alpha,L)-\Theta(\alpha,\D)
 \ge\pi\alpha\left(8a e^{-\pi\alpha a}
       -\sum_{v\in\D}|v|^2e^{-\pi\alpha|v|^2}\right).
\]
The right-hand side is positive for $\alpha\ge5$ in the remaining
length range. For $1\le\alpha\le5$, project along a shortest vector
$v$ onto $v^\perp$:
\[
 \Gamma_a=a^{-1/2}\pi_{v^\perp}(L),\qquad
 \covol(\Gamma_a)=a^{-2}.
\]
The stationarity identities express theta through weighted sums
over the complete fibres. For a radial Schwartz minorant $f$ of
these fibre weights with $\widehat f\ge0$, Poisson summation gives
\[
 \sum_{y\in\Gamma_a\setminus\{0\}}f(y)
 \ge a^2\widehat f(0)-f(0).
\]
A finite family of parameter-dependent minorants, together with
the moment bound, yields strict comparison on the remaining
rectangle. Applying this comparison at a global minimum gives
$\min\Phi=0$, since $\Phi(\alpha,\D)=0$. Every equality pair is
then a global minimum, so the same stationary comparison gives
uniqueness. Appendix~\ref{sec:verification} proves the uniform
auxiliary estimates by finite interval bounds and analytic tail
estimates; Appendix~\ref{app:coefficients} specifies the exact
rational coefficients.

In Section~\ref{sec:epstein}, Mellin integration of the theta
comparison gives, for $s>2$,
\[
 E(L,s)-E(\D,s)=\frac{\pi^s}{\Gamma(s)}
 \int_0^\infty\alpha^{s-1}
 \bigl(\Theta(\alpha,L)-\Theta(\alpha,\D)\bigr)\,\dd\alpha\ge0.
\]
Strictness at every scale yields uniqueness, and positive
superposition gives the same conclusion for Gaussian mixtures.

\section{Determinant separation}
\label{sec:reduction}\label{sec:geometry}
In this section, we prove a determinant bound outside the trace
neighborhoods of $D_4$ and $A_4$ defined in
Subsection~\ref{subsec:reduction-reference}. After normalization to
covolume one, this gives an upper bound for the shortest squared
length outside these neighborhoods.
\begin{theorem}[Determinant separation]
\label{prop:barrier}
Fix $r\in(0,2)$ and $s\in(0,2/3)$. If $R\in\Ry$ satisfies
\[
 \inf_{P\in\DD}h_P(R)\ge r,
 \qquad \inf_{P\in\Acl}h_P(R)\ge s,
\]
then
\begin{equation}\label{eq:barrier}
 \det R\ge\min\left\{\frac{(1-r/2)^2(1+r)^2}{4},
              \frac5{16}(1-3s/2)(1+5s/2)\right\}.
\end{equation}
\end{theorem}

Reduction ensures that the determinant minimum is attained.
Concavity reduces the minimum to a vertex of the truncated Ryshkov
domain. Every such vertex lies on a trace boundary, where the
tangent-cone estimates give the bound.

\begin{lemma}[Compactness of reduced forms]
\label{lem:compact}
Fix $\varepsilon>0$ and $D_0>0$. The Minkowski-reduced matrices
$Q\in\Sym_4^+$ satisfying
\[
 q_{11}\ge\varepsilon,\qquad \det Q\le D_0
\]
form a compact subset of $\Sym_4^+$. In particular, the determinant-one
reduced forms with $a(Q)\ge\varepsilon$ form a compact set.
\end{lemma}
\begin{proof}
The ordering of the diagonal entries and
\eqref{eq:Barnes} give
\[
 q_{ii}\le\frac{4D_0}{\eps^3},\qquad |q_{ij}|\le\sqrt{q_{ii}q_{jj}}\le\frac{4D_0}{\eps^3},
 \qquad \det Q\ge\frac{\eps^4}{4}.
\]
The entry bounds give a uniform upper bound for the eigenvalues;
the positive determinant bound then gives a uniform lower bound.
The reduction inequalities are closed, so every sequence has a
subsequence converging to a matrix in the same set.
\end{proof}

\label{subsec:tangentgeometry}
The minimum-one $D_4$ lattice has unit roots
\[
 \mathcal V=\{(\pm e_i\pm e_j)/\sqrt2:1\le i<j\le4\}.
\]
Define
\begin{equation}\label{eq:tangent}
 \T=\{C=C^T:\tr C=1,\ c_{ii}+c_{jj}\ge2|c_{ij}|
                                  \ (1\le i<j\le4)\}.
\end{equation}
\begin{lemma}[The $D_4$ tangent section]
\label{lem:tangent}
The trace-one polytope $\T$ in \eqref{eq:tangent} has $64$ vertices:
$48$ have spectrum $(-1/2,0,0,3/2)$ and $16$ have spectrum
$(-1/2,-1/2,1,1)$. Every $C\in\T$ satisfies
\begin{equation}\label{eq:tangentbounds}
 -\tfrac12I\preceq C\preceq\tfrac32I,
 \qquad 0\le v^TCv\le1\quad(v\in\mathcal V).
\end{equation}
Moreover, for every $r\in[0,2)$,
\begin{equation}\label{eq:detD}
 (1-r/2)^2(1+r)^2\le\det(I+rC)\le(1+r/4)^4.
\end{equation}
\end{lemma}
\begin{proof}
The diagonal entries satisfy $\sum_i c_{ii}=1$ and
$c_{ii}+c_{jj}\ge0$. If they are all nonnegative, the diagonal is a
convex combination of the coordinate vectors.
Otherwise its unique negative entry is, say, $c_{11}=-s$, with
$0<s\le1/2$. The other entries are at least $s$; subtracting
$2s(-1/2,1/2,1/2,1/2)$ leaves nonnegative entries of sum $1-2s$.
The diagonal polytope thus has the eight vertices obtained by permuting
$(1,0,0,0)$ and $(-1/2,1/2,1/2,1/2)$.

The off-diagonal entries independently fill
$[-(c_{ii}+c_{jj})/2,(c_{ii}+c_{jj})/2]$, with endpoints affine in
the diagonal. Decompose the diagonal while keeping these relative
positions fixed; the full polytope is then the convex hull of the
endpoint fibres over the eight diagonal vertices. Each fibre has three
nondegenerate intervals, giving $8\cdot2^3=64$ vertices, all extreme
because both the diagonal and the fibre point are extreme.

For diagonal $(1,0,0,0)$, the off-diagonal entries form a star with
weights $\pm1/2$; each vertex matrix has nonzero eigenvalues $-1/2$
and $3/2$, accounting for 32 vertices.
In the other pattern, the negative coordinate is isolated and the
remaining block is a signed triangle. Write
$\boldsymbol1_3=(1,1,1)^T$, so $\boldsymbol1_3\boldsymbol1_3^T$
is the $3\times3$ all-ones matrix. Up to diagonal orthogonal
conjugation, the block is $\frac12\boldsymbol1_3\boldsymbol1_3^T$
or $I_3-\frac12\boldsymbol1_3\boldsymbol1_3^T$ according as the
edge-sign product is positive or negative, with sixteen vertices each.

Convexity gives the spectral bounds. For a root supported on $i,j$,
$v^TCv=(c_{ii}+c_{jj})/2\pm c_{ij}$ lies in $[0,c_{ii}+c_{jj}]$;
the complementary pair gives $c_{ii}+c_{jj}\le1$.
For $r<2$, $I+rC$ is positive definite. Concavity of $\log\det$
reduces the lower bound to the two vertex spectra; the smaller
determinant is $(1-r/2)^2(1+r)^2$. The upper bound follows from
the arithmetic--geometric mean inequality and $\tr C=1$.
\end{proof}

For $P=B^TB\in\DD$, take $B$ as a basis of $2^{-1/2}D_4$
and set $M=B^{-T}RB^{-1}$. The root constraints $v^TMv\ge1$
and $\sum_{v\in\mathcal V}vv^T=6I$ give $h_P(R)\ge0$.
Equality makes every root value of $M-I$ vanish; the two signs on
each coordinate pair then force $M=I$. For $r=h_P(R)>0$,
\begin{equation}\label{eq:physicalcap}
 M=I+rC,\quad C\in\T,\quad
 \det R=\tfrac14\det M.
\end{equation}
If $R=Q/a(Q)$ and $\det Q=1$, the represented covolume-one lattice is
isometric to $A^{1/2}\D$, where
\begin{equation}\label{eq:Anormalization}
 A=\frac{M}{(\det M)^{1/4}}.
\end{equation}
Indeed, $Q=a(Q)B^TMB$ and $\det B^TB=1/4$ give
$a(Q)=\sqrt2(\det M)^{-1/4}$. Thus
\[
 Q=(2^{1/4}B)^TA(2^{1/4}B),
\]
where $2^{1/4}B$ is a basis of $\D$. In particular,
$\det A=1$ and $\tr\log A=0$.

At $A_4$, the trace section is a simplex. For the representative
$P_A$ in \eqref{eq:minimum-one-forms},
\[
 P_A^{-1}=2I_4-\tfrac25\boldsymbol1\boldsymbol1^T.
\]
In the basis $2^{-1/2}B_A$, the ten antipodal root pairs have
integer coordinate representatives $e_i$ and $e_i-e_j$.
The tangent cone is given by $K=K^T$, $K[e_i]\ge0$ and
$K[e_i-e_j]\ge0$. The coordinates $y_i=K[e_i]$ and
$y_{ij}=K[e_i-e_j]$ determine $K$ through
$K_{ij}=(y_i+y_j-y_{ij})/2$, and
\begin{equation}\label{eq:Atrace}
 \tr(P_A^{-1}K)=\frac25\left(\sum_i y_i+\sum_{i<j}y_{ij}\right).
\end{equation}
The trace-one section is therefore the simplex with one
$y$-coordinate equal to $5/2$ at each vertex and all others zero.

\begin{lemma}[The $A_4$ tangent determinant]
\label{lem:Atangent}
At each vertex $K$ of the trace-one $A_4$ tangent simplex,
$P_A^{-1/2}KP_A^{-1/2}$ has spectrum $(-3/2,0,0,5/2)$. For every $R\in\Ry$
and $P\in\Acl$, one has $h_P(R)\ge0$, with equality precisely when
$R=P$. If $s=h_P(R)<2/3$, then
\begin{equation}\label{eq:detA}
 \det R\ge\frac5{16}(1-3s/2)(1+5s/2).
\end{equation}
\end{lemma}
\begin{proof}
Equation~\eqref{eq:Atrace} gives nonnegativity and its equality case.
Permutations of the five coordinates of $A_4$ act transitively on
the ten antipodal root pairs, and hence on the vertices of the
tangent simplex. It suffices to take $y_1=5/2$ and all other
$y$-coordinates zero; the characteristic polynomial of $P_A^{-1}K$
is $z^2(z+3/2)(z-5/2)$. For $s=0$, equality holds. For $s>0$,
$K=(R-P)/s$ has nonnegative root evaluations and $\tr(P^{-1}K)=1$,
so lies in the tangent simplex after an integral basis change.
Since $I+sP^{-1/2}KP^{-1/2}$ is positive definite for $s<2/3$,
concavity of $\log\det$ gives the bound after multiplication by
$\det P=5/16$.
\end{proof}

\begin{proof}[Proof of Theorem~\ref{prop:barrier}]
Let $\mathcal C\subset\Ry$ be the convex, integrally invariant set
defined by the trace inequalities. The fixed determinants of
$P^{-1}$ give a uniform positive lower bound for $\tr P^{-1}$, so
$\mathcal C$ contains a sufficiently large scalar matrix.
The set $\mathcal C$ and the determinant are invariant under
$R\mapsto U^TRU$, $U\in\GL_4(\Z)$, since both $\DD$ and $\Acl$
are full integral orbits. We may therefore choose a
determinant-minimizing sequence of reduced forms. Its determinants
are bounded and its minima are at least one, so
Lemma~\ref{lem:compact} gives a positive definite subsequential limit;
closedness places it in $\mathcal C$, attaining the infimum.

The constraints are locally finite. Near $R_0\in\Sym_4^+$,
restrict to $R\succeq R_0/2$. A trace constraint can be active or
violated only if $\tr(P^{-1}R_0)\le2(4+\max\{r,s\})$.
Since $R_0\succeq\varepsilon I$ for some $\varepsilon>0$, this bounds
$\tr P^{-1}$ and every entry of $P^{-1}$. For $P=U^TP_0U$, the matrices
$P^{-1}=U^{-1}P_0^{-1}U^{-T}$ lie in a fixed rational lattice,
so only finitely many occur. Likewise, active or violated Ryshkov
constraints require $R_0[m]\le2$, allowing only finitely many $m$.
Thus $\mathcal C$ and $\Ry$ are locally polyhedral. Constraints strict
at $R_0$ remain strict on a smaller neighborhood.

Strict concavity of $\log\det$ implies that every minimizer is a
vertex of $\mathcal C$. At a point where all trace inequalities
are strict, local finiteness gives the same local feasible directions
in $\mathcal C$ and $\Ry$. Thus a vertex of $\mathcal C$ is either
a vertex of $\Ry$ or lies on a trace boundary. Every vertex
$P\in\DD\cup\Acl$ of $\Ry$ is removed by its own trace inequality,
since $h_P(P)=0$. The minimizing vertex $R_0$ therefore lies on a
trace boundary. In particular, either
$h_P(R_0)=r$ for some $P\in\DD$ or $h_P(R_0)=s$ for some $P\in\Acl$.
Lemma~\ref{lem:tangent} or Lemma~\ref{lem:Atangent} bounds
$\det R_0$ from below by the corresponding expression in
\eqref{eq:barrier}; minimality gives the result for every
$R\in\mathcal C$.
\end{proof}

\begin{remark}
The perfect-form classification enters only through the determinant
bound. The theta estimates apply to arbitrary lattice shapes.
\end{remark}

\section{Local comparison and localization}
\label{sec:local}\label{sec:preliminary}\label{sec:short}
\label{subsec:perfect-comparison}
\begin{theorem}[Local comparison at $D_4$]
\label{prop:Dcap}
Let $Q\in\Sym_4^+$ have determinant one, and set $R=Q/a(Q)$.
Suppose that $h_P(R)\le11/25$ for some $P\in\DD$.
Then
\begin{equation}\label{eq:Dcap}
 \Theta(\alpha,Q)\ge\Theta(\alpha,\D)\qquad(\alpha\ge1).
\end{equation}
For each fixed $\alpha\ge1$, equality holds precisely when the lattice
represented by $Q$ is isometric to $\D$.
More precisely, let $A$ be the determinant-one relative deformation
associated with $P$ and $Q$ in \eqref{eq:Anormalization}. For every
$\alpha\ge1$,
\begin{equation}\label{eq:capcoercivity}
 \Theta(\alpha,Q)-\Theta(\alpha,\D)
 \ge\frac{\pi\sqrt2}{440}\|\log A\|_F^2\alpha e^{-\pi\sqrt2\alpha}.
\end{equation}
Here $\|\cdot\|_F$ denotes the Frobenius norm.
\end{theorem}

The second and fourth shell moments, combined with a quadratic
minorant of the exponential, give lower bounds in terms of the
eigenvalues of $A$. The tangent-cone estimates sharpen
the first-shell bound, yielding the quantitative gap, while every
higher shell contributes a nonnegative difference.

The nonzero squared lengths of $\D$ are positive integral multiples
of $\sqrt2$; its first shell consists of the $24$ vectors
$2^{-1/4}(\pm e_i\pm e_j)$, $i<j$.

\begin{lemma}[Shell moment identities]
\label{lem:moments}
Let $\rho>0$ be a squared length represented by $\D$, and put
$S_\rho=\{v\in\D:|v|^2=\rho\}$.
For every real symmetric matrix $K$,
\begin{align}
 \frac1{|S_\rho|}\sum_{v\in S_\rho}\frac{v^TKv}{\rho}
 &=\frac{\tr K}{4},\label{eq:moment2}\\
 \frac1{|S_\rho|}\sum_{v\in S_\rho}
                  \left(\frac{v^TKv}{\rho}\right)^2
 &=\frac{(\tr K)^2+2\tr K^2}{24}.\label{eq:moment4}
\end{align}
\end{lemma}
\begin{proof}
Each shell is invariant under coordinate permutations, sign changes,
and the orthogonal Hadamard matrix
\[
 \frac12\begin{pmatrix}
 1&1&1&1\\1&1&-1&-1\\1&-1&1&-1\\1&-1&-1&1
 \end{pmatrix},
\]
which preserves $D_4$ and $\D$. For
$A=\sum v_1^4$ and $B=\sum v_1^2v_2^2$ over $S_\rho$,
sign symmetry and Hadamard invariance give $A=(4A+36B)/16$,
hence $A=3B$.
Also $4A+12B=|S_\rho|\rho^2$, so
$A=|S_\rho|\rho^2/8$ and $B=|S_\rho|\rho^2/24$.
The same symmetries give
$\sum v_i v_j=(|S_\rho|\rho/4)\delta_{ij}$, proving
\eqref{eq:moment2}. For $K=(k_{ij})$, the fourth-moment expansion is
\[
 \sum_{v\in S_\rho}(v^TKv)^2
 =A\sum_i k_{ii}^2+2B\sum_{i<j}k_{ii}k_{jj}
                         +4B\sum_{i<j}k_{ij}^2.
\]
By $A=3B$ and $\tr K^2=\sum_i k_{ii}^2+2\sum_{i<j}k_{ij}^2$,
the right-hand side equals $B((\tr K)^2+2\tr K^2)$. Division by
$|S_\rho|\rho^2$ gives \eqref{eq:moment4}.
\end{proof}

The next lemma bounds an exponential sum in terms of its first two
moments. We apply it to shell displacements with equal weights.
\begin{lemma}[An exponential bound from two moments]
\label{lem:two-moment}
Let $\delta_1,\ldots,\delta_N\le b$ be real numbers, and let
$\omega_j\ge0$ satisfy $\sum_{j=1}^N\omega_j=1$. Set
\[
 m=\sum_{j=1}^N\omega_j\delta_j,\qquad
 V=\sum_{j=1}^N\omega_j(\delta_j-m)^2,\qquad \eta=b-m,
\]
and assume that $m<b$ and $V>0$. Then, for every $s>0$,
\begin{equation}\label{eq:HermiteH}
 \sum_{j=1}^N\omega_j e^{-s\delta_j}\ge
 \Hh_s(m,V,\eta):=
 e^{-sm}\frac{\eta^2e^{sV/\eta}+Ve^{-s\eta}}{\eta^2+V}.
\end{equation}
\end{lemma}
\begin{proof}
Set $z=m-V/\eta$ and let
\[
 p=\frac{\eta^2}{\eta^2+V},\qquad q=\frac{V}{\eta^2+V}.
\]
Then $p+q=1$, $pz+qb=m$, and
$p(z-m)^2+q(b-m)^2=V$. Consequently,
$\sum_{j=1}^N\omega_j Q(\delta_j)=pQ(z)+qQ(b)$
for every polynomial $Q$ of degree at most two.

Let $q_2$ be the quadratic interpolant to $f(t)=e^{-st}$ at $z,b$,
with derivative $f'$ at $z$. The Hermite remainder gives
\[
 f(t)-q_2(t)
 =\frac{f^{(3)}(\xi)}{3!}(t-z)^2(t-b)\ge0,
 \qquad t\le b,
\]
with equality at the nodes; here $f^{(3)}(\xi)=-s^3e^{-s\xi}$.
Summing at $\delta_j$ with coefficients $\omega_j$ gives
\[
 \sum_{j=1}^N\omega_j f(\delta_j)
 \ge\sum_{j=1}^N\omega_j q_2(\delta_j)
 =p f(z)+q f(b)=\Hh_s(m,V,\eta).\qedhere
\]
\end{proof}

For $A\in\Sym_4^+$ with determinant one, put
\begin{equation}\label{eq:spectralstats}
 \overline\lambda=\tfrac14\tr A,\quad m=\overline\lambda-1,
 \quad V=\tfrac1{12}\sum_{i=1}^4(\lambda_i-\overline\lambda)^2.
\end{equation}
Here $\lambda_i$ are the eigenvalues of $A$. The determinant constraint
gives $m\ge0$, with equality only at $A=I_4$, and $V>0$ when
$A\ne I_4$. At $A=I_4$, all displacements vanish and we define
$\Hh_s=1$.
For a shell $S_\rho$, put $\delta_v=v^T(A-I_4)v/\rho$.
Lemma~\ref{lem:moments} gives
\[
 \frac1{|S_\rho|}\sum_{v\in S_\rho}\delta_v=m,\qquad
 \frac1{|S_\rho|}\sum_{v\in S_\rho}(\delta_v-m)^2=V.
\]
Also,
\[
 \frac1{|S_\rho|}\sum_{v\in S_\rho}e^{-\pi\alpha v^TAv}
 =\frac{e^{-\pi\alpha\rho}}{|S_\rho|}
   \sum_{v\in S_\rho}e^{-s\delta_v},\qquad s=\pi\alpha\rho.
\]
For $A\ne I_4$, Lemma~\ref{lem:two-moment} applies with weights
$1/|S_\rho|$ and $b=\lambda_{\max}(A)-1$. On the first shell, root
constraints sharpen the bound for $\delta_v$ in the fixed coordinates.

\begin{lemma}[A moment bound near the identity]
\label{lem:smallband}
Let $A\in\Sym_4^+$ have determinant one and spectrum in
$[9/10,11/10]$, and let $m,V$ be defined by \eqref{eq:spectralstats}.
If $A\ne I_4$ and $m<b\le1/10$, set $\eta=b-m$. Then
\begin{equation}\label{eq:smallband}
 \Hh_{22/5}(m,V,\eta)\ge1+\frac{22}{61}m.
\end{equation}
For $A=I_4$, the left-hand side is defined to be $1$.
\end{lemma}
\begin{proof}
For $x\ge9/10$, integration gives
$x-1-\log x\le(x-1)^2/(2(9/10))$.
Summing over the eigenvalues and using $\log\det A=0$ yields
$\tr(A-I)^2\ge(36/5)m$, hence, by \eqref{eq:spectralstats},
$V+m^2\ge3m/5+2m^2/3$.
Apply the elementary inequality
$e^{-u}\ge1-u+u^2/(2+B)$, valid for $u\le B$ and $B\ge0$,
with $B=11/25$ at $u=(22/5)(m-V/\eta)$ and $u=(22/5)b$.
Multiplying by $\eta^2/(\eta^2+V)$ and $V/(\eta^2+V)$, respectively,
and adding gives
\[
 \Hh_{22/5}\ge1-\frac{22}{5}m
       +\frac{(22/5)^2}{2+11/25}(V+m^2)
 \ge1+\frac{22}{61}m.
\]
\end{proof}

\begin{lemma}[The higher-shell moment estimate]
\label{lem:highband}
Let positive numbers $\lambda_1,\ldots,\lambda_4$ satisfy
$2/3\le\lambda_i\le5/3$ and $\prod_{i=1}^4\lambda_i=1$. Set
\[
 \overline\lambda=\frac14\sum_{i=1}^4\lambda_i,\qquad
 m=\overline\lambda-1,\qquad
 V=\frac1{12}\sum_{i=1}^4(\lambda_i-\overline\lambda)^2,\qquad
 \eta=\max_i\lambda_i-\overline\lambda.
\]
Then
\begin{equation}\label{eq:highband}
 \Hh_{44/5}(m,V,\eta)\ge1+\frac{m}{10}.
\end{equation}
At $(\lambda_1,\ldots,\lambda_4)=(1,1,1,1)$, the left-hand side
is defined to be $1$.
\end{lemma}
The proof of Lemma~\ref{lem:highband} is given in
Appendix~\ref{subsec:spectral-checks}.

To sharpen the root-shell bound for $\delta_v$, let $M=I+rC$,
$C\in\T$, and $0\le r\le11/25$. For
$w_i=\lambda_i(M)/(1+r/4)$, we have
\begin{equation}\label{eq:wrelax}
 \sum_iw_i=4,\qquad
 \frac{26}{37}\le w_i\le\frac{166}{111},\qquad
 \prod_iw_i\ge\frac{1557504}{1874161}.
\end{equation}
The eigenvalue bounds follow from Lemma~\ref{lem:tangent}.
By \eqref{eq:detD},
\[
 \prod_iw_i\ge\frac{(1-r/2)^2(1+r)^2}{(1+r/4)^4}.
\]
The right-hand side decreases on $[0,11/25]$ to
$1557504/1874161$, proving \eqref{eq:wrelax}.
Let
\begin{equation}\label{eq:rootstats}
 \begin{aligned}
 \mu&=(w_1w_2w_3w_4)^{-1/4},& m&=\mu-1,\\
 V&=\frac{\mu^2}{12}\sum_i(w_i-1)^2,&
 \eta&=\mu\bigl(\min(u,48/37)-1\bigr),\qquad u=\max_iw_i.
 \end{aligned}
\end{equation}
The normalized eigenvalues are $\mu w_i$. Applying
Lemma~\ref{lem:tangent} in the fixed root directions gives
\[
 \frac{v^TAv}{|v|^2}\le
 \mu\min\left(u,\frac{1+r}{1+r/4}\right)
 \le\mu\min(u,48/37).
\]
Thus the shell identities hold with $m,V$ from \eqref{eq:rootstats},
and $\delta_v\le m+\eta$ on the first shell.

\begin{lemma}[The root-shell moment estimate]
\label{lem:rootcert}
Let $w=(w_1,w_2,w_3,w_4)$ satisfy \eqref{eq:wrelax}, and define
$m,V,\eta$ by \eqref{eq:rootstats}. Then
\begin{equation}\label{eq:rootcert}
 \Hh_{22/5}(m,V,\eta)\ge1+\frac{m}{200}.
\end{equation}
At $w=(1,1,1,1)$, the left-hand side is defined by continuity to be $1$.
\end{lemma}
For $w\ne(1,1,1,1)$, one has $m,V,\eta>0$. As $q=u-1\to0$,
the constraints give $w_i-1=O(q)$, $m,V=O(q^2)$, and $\eta\sim q$.
Thus both nodes in \eqref{eq:HermiteH} tend to zero and $\Hh_s\to1$,
since their nonnegative weights sum to one.
The quantitative estimate is proved in Appendix~\ref{subsec:root-checks}.

\begin{proof}[Proof of Theorem~\ref{prop:Dcap}]
Write $r=h_P(R)$. The normalization \eqref{eq:Anormalization}
and $a(R)=a(P)=1$ give
\[
 A=I_4\quad\Longleftrightarrow\quad R=P
 \quad\Longleftrightarrow\quad r=0.
\]
If $r=0$, both assertions follow. Henceforth assume $r>0$, so
$A\ne I_4$ and $m,V>0$. Since $Q$ represents $A^{1/2}\D$,
\[
 \Theta(\alpha,Q)=\sum_{v\in\D}e^{-\pi\alpha v^TAv}.
\]
For $0\le r\le11/25$, one has $(1-r/2)^2(1+r)^2\ge1$.
The determinant bounds in \eqref{eq:detD} therefore give
\[
 \lambda_{\min}(A)\ge\frac{1-r/2}{1+r/4}\ge\frac{26}{37}>\frac23,
 \qquad \lambda_{\max}(A)\le1+3r/2\le\frac{83}{50}<\frac53.
\]
For fixed $m,V,\eta$, $s\mapsto\Hh_s$ is convex with $\Hh_0=1$.
Thus $(\Hh_s-1)/s$ is nondecreasing for $s>0$, giving
$\Hh_s-1\ge(s/s_0)(\Hh_{s_0}-1)$ for $s\ge s_0>0$.
On the first shell, use $\eta$ from \eqref{eq:rootstats} and
$s_0=22/5$. Lemmas~\ref{lem:two-moment} and~\ref{lem:rootcert}
give a relative excess of at least $sm/880$, where
$s=\pi\sqrt2\alpha>22/5$. On every higher shell, use
$\eta=\lambda_{\max}(A)-\overline\lambda$ and $s_0=44/5$.
Since $s=\pi\alpha\rho\ge2\pi\sqrt2\alpha>44/5$,
Lemmas~\ref{lem:two-moment} and~\ref{lem:highband} give a
nonnegative shell difference. Absolute convergence permits
summation over all shells; the first shell alone gives
\[
 \Theta(\alpha,Q)-\Theta(\alpha,\D)
 \ge24e^{-\pi\sqrt2\alpha}\frac{\pi\sqrt2\alpha m}{880}>0.
\]
For the quantitative bound, each eigenvalue $h$ of $\log A$ satisfies
\[
 e^h-1-h
 =h^2\int_0^1(1-u)e^{uh}\,\dd u\ge\frac{h^2}{3},
\]
since $e^{uh}\ge2/3$ for $0\le u\le1$.
Summing and using $\tr\log A=0$ yields
$m\ge\|\log A\|_F^2/12$. Substitution in the preceding theta bound
gives \eqref{eq:capcoercivity}, and strictness for $r>0$ gives the
equality statement.
\end{proof}

For the remaining comparisons and the global argument, we record
the Fourier identities, compactness statement, and reference estimates.
We use the Fourier convention
\[
 \widehat f(\xi)=\int_{\R^n}f(x)e^{-2\pi i x\cdot\xi}\dd x.
\]
Poisson summation and the Gaussian Fourier transform give
\[
 \sum_{v\in\Lambda}f(v)=\frac1{\covol(\Lambda)}
 \sum_{\xi\in\Lambda^*}\widehat f(\xi),\qquad
 \widehat{e^{-\pi\alpha|\cdot|^2}}(\xi)
 =\alpha^{-n/2}e^{-\pi|\xi|^2/\alpha}
\]
for Schwartz functions $f$. In particular,
\begin{equation}\label{eq:poisson}
 \Theta(\alpha,L)=\alpha^{-2}\Theta(1/\alpha,L^*)
 \qquad(\covol L=1).
\end{equation}
We also use Poisson summation in dimensions one and three.

\begin{lemma}[Isoduality and reciprocal scales]
\label{lem:isodual}
The matrix
\begin{equation}\label{eq:J}
 J=\frac1{\sqrt2}
 \begin{pmatrix}
 1&1&0&0\\1&-1&0&0\\0&0&1&1\\0&0&1&-1
 \end{pmatrix}
\end{equation}
is an orthogonal involution and satisfies
\begin{equation}\label{eq:normalizeddual}
 \D^*=2^{1/4}D_4^*=J\D.
\end{equation}
Consequently, for every covolume-one lattice $L\subset\R^4$ and every
$\alpha>0$,
\begin{equation}\label{eq:defectduality}
 \Phi(\alpha,L)=\Phi(1/\alpha,L^*).
\end{equation}
\end{lemma}
\begin{proof}
By \eqref{eq:D-dual-model}, parity gives
$JD_4\subset\sqrt2D_4^*$. The two lattices have covolume $2$, so
the inclusion is an equality. Since $J^TJ=J^2=I_4$, the matrix $J$
is an orthogonal involution.
Rescaling gives $J\D=2^{1/4}D_4^*=\D^*$.
Subtract \eqref{eq:poisson} for $L$ and $\D$ and multiply by $\alpha$
to obtain \eqref{eq:defectduality}.
\end{proof}

\begin{lemma}[Compact theta sublevels and smooth dependence]
\label{lem:sublevels}
Fix $0<b<B<\infty$ and $C>0$. Among covolume-one lattices, the condition
\[
 \Theta(\alpha,L)\le C\quad\text{for some }\alpha\in[b,B]
\]
implies a uniform positive lower bound for $a(L)$, and the corresponding
reduced Gram matrices lie in a compact subset of $\Sym_4^+$.
On compact subsets of $\Sym_4^+\times(0,\infty)$, the theta series
and its termwise derivatives of every fixed order in the Gram entries
and scale converge uniformly. The function $a(Q)$ is continuous
on $\Sym_4^+$.
\end{lemma}
\begin{proof}
For a shortest vector $v$ with $|v|^2=a$, summing over $\Z v$ gives
\[
 \Theta(\alpha,L)\ge\sum_{k\in\Z}e^{-\pi\alpha a k^2}
 =\frac1{\sqrt{\alpha a}}
   \sum_{k\in\Z}e^{-\pi k^2/(\alpha a)}
 \ge\frac1{\sqrt{\alpha a}}.
\]
Thus $a\ge1/(BC^2)$, and Lemma~\ref{lem:compact} gives compactness.

On a compact subset of $\Sym_4^+\times(0,\infty)$, choose
$0<\lambda\le\Lambda$ and $0<b<B$ with
$\lambda I\preceq Q\preceq\Lambda I$ and $b\le\alpha\le B$.
Termwise derivatives of any fixed order have a summable majorant
\[
 C(1+|m|)^N e^{-\pi b\lambda|m|^2},
\]
uniformly in $Q,\alpha$. This justifies uniform convergence and
termwise differentiation, including on the determinant-one constraint.

Since $a(Q)\le\Lambda$ and $Q[m]\ge\lambda|m|^2$,
\[
 a(Q)=\min_{\substack{m\in\Z^4\setminus\{0\}\\
                         |m|^2\le\Lambda/\lambda}}Q[m],
\]
a finite minimum of continuous functions; hence $a(Q)$ is continuous.
\end{proof}

The four-square formula gives the reference expansion~\cite{CS}
\begin{equation}\label{eq:Dseries}
 \Theta(\alpha,\D)=1+24\sum_{n\ge1}\sodd(n)e^{-\pi\sqrt2\alpha n},
 \qquad \sodd(n)=\sum_{\substack{k\mid n\\k\text{ odd}}}k.
\end{equation}
An integer vector belongs to $D_4$ exactly when its squared norm is even. The
four-square formula $8\sum_{k\mid2n,\,4\nmid k}k=24\sodd(n)$
therefore gives \eqref{eq:Dseries} after normalization.
The bound $\sodd(n)\le n^2$ controls the tails.

We use the Gaussian moment
\begin{equation}\label{eq:basicnotation}
 M(t,L)=\sum_{v\in L}|v|^2e^{-t|v|^2},\qquad t>0.
\end{equation}

\begin{lemma}[Reference energy and moment bounds]
\label{lem:reference}
For every $\alpha\ge1$, the reference lattice satisfies
\begin{equation}\label{eq:refbound}
 e^{\pi\sqrt2\alpha}\bigl(\Theta(\alpha,\D)-1\bigr)<\frac{243}{10}.
\end{equation}
For every $\alpha\ge5$, it also satisfies
\begin{equation}\label{eq:MD34}
 e^{\pi\sqrt2\alpha}M(\pi\alpha,\D)<34.
\end{equation}
\end{lemma}
See Appendix~\ref{subsec:reference-checks} for the proof.

\begin{proposition}[Comparison near $A_4$]
\label{prop:Acap}
Let $Q\in\Sym_4^+$ have determinant one and satisfy
$h_P(Q/a(Q))\le1/10$ for some $P\in\Acl$. Then
\begin{equation}\label{eq:Acap}
 \Theta(\alpha,Q)-\Theta(\alpha,\D)
 >\frac1{10}e^{-\pi\sqrt2\alpha}\qquad(\alpha\ge1).
\end{equation}
\end{proposition}
\begin{proof}
After an integral change of basis, we may assume $P=P_A$. Write
$a=a(Q)$, $R=Q/a$, and $s=h_{P_A}(R)$. The twenty distinct integer
vectors
\[
 \mathcal V_A=\{\pm e_i:1\le i\le4\}
       \cup\{\pm(e_i-e_j):1\le i<j\le4\}
\]
satisfy
\[
 \sum_{v\in\mathcal V_A}vv^T
 =10I_4-2\boldsymbol1\boldsymbol1^T=5P_A^{-1}.
\]
Their average squared length in $Q$ is
\[
 \frac1{20}\sum_{v\in\mathcal V_A}Q[v]
 =\frac14\tr(P_A^{-1}Q)=a(1+s/4).
\]
Since $a=(\det R)^{-1/4}$, Lemma~\ref{lem:Atangent} gives
\[
 a(1+s/4)\le
 \frac{2(1+s/4)}{5^{1/4}[(1-3s/2)(1+5s/2)]^{1/4}}
 <\frac{13503}{10000},\qquad 0\le s\le\frac1{10}.
\]
Jensen's inequality therefore gives
\[
 \Theta(\alpha,Q)>1+20e^{-\pi\alpha(13503/10000)}.
\]
Comparison with \eqref{eq:refbound} yields \eqref{eq:Acap} for every
$\alpha\ge1$.
\end{proof}

Determinant separation also gives a length bound outside the two
trace neighborhoods.
\begin{corollary}[Length bound away from the perfect forms]
\label{cor:abound}
Let $Q\in\Sym_4^+$ have determinant one, and set $R=Q/a(Q)$.
If
\[
 h_P(R)\ge\frac{11}{25}\quad(P\in\DD),\qquad
 h_P(R)\ge\frac1{10}\quad(P\in\Acl),
\]
then
\begin{equation}\label{eq:amax}
 a(Q)\le\frac{25}{\sqrt{351}}.
\end{equation}
\end{corollary}
\begin{proof}
Theorem~\ref{prop:barrier} with $r=11/25$, $s=1/10$ gives
$\det R\ge351^2/25^4$. Since $\det R=a(Q)^{-4}$, this is
\eqref{eq:amax}.
\end{proof}

For the short-vector comparison, we use the normalized four-dimensional
Laguerre basis
\[
 \ell_k^{(4)}(u)=\frac{L_k^1(u)}{k+1}
 =\sum_{j=0}^k\frac{(-1)^j\binom{k+1}{k-j}}{j!(k+1)}u^j.
\]
If $P(u)=\sum b_k\ell_k^{(4)}(u)$ and
$P^\sharp(u)=\sum(-1)^kb_k\ell_k^{(4)}(u)$, then
\begin{equation}\label{eq:Lag4}
 \widehat{\bigl[e^{-\pi\alpha|x|^2}P(2\pi\alpha|x|^2)\bigr]}(y)
 =\alpha^{-2}e^{-\pi|y|^2/\alpha}P^\sharp(2\pi|y|^2/\alpha).
\end{equation}
The Gaussian transform and Laguerre generating function give
$e^{-\pi|x|^2}L_k^{n/2-1}(2\pi|x|^2)$ the Fourier eigenvalue $(-1)^k$
in dimension $n$. Normalization and dilation yield \eqref{eq:Lag4}
and, for $n=3$, \eqref{eq:FT3}.

\begin{proposition}[Exclusion by a short vector]
\label{prop:cusp}
For every covolume-one lattice $L\subset\R^4$,
\begin{equation}\label{eq:cusp}
 a(L)\le\frac{21}{20}\quad\Longrightarrow\quad
 \Theta(\alpha,L)>\Theta(\alpha,\D)\qquad(\alpha\ge1).
\end{equation}
Under the same short-vector hypothesis, the stronger estimate
\[
 \Theta(\alpha,L)-\Theta(\alpha,\D)>\frac1{4000}
 \qquad\left(1\le\alpha\le\frac94\right)
\]
holds.
\end{proposition}
\begin{proof}
Let $S=1-P$. If $S\ge0$ and $P^\sharp\ge0$ on the
interval $[0,\infty)$, Poisson summation and \eqref{eq:Lag4} give
\begin{align}\label{eq:cuspidentity}
 \Theta(\alpha,L)
 ={}&\sum_{v\in L}e^{-\pi\alpha|v|^2}S(2\pi\alpha|v|^2)\notag\\
 &+\alpha^{-2}\sum_{y\in L^*}
       e^{-\pi|y|^2/\alpha}P^\sharp(2\pi|y|^2/\alpha).
\end{align}
If $S-2S'\ge0$ on $[0,(21\pi/10)\alpha]$, then
$u\mapsto e^{-u/2}S(u)$ is nonincreasing there. Retaining the origin
in each sum and a shortest pair in the first, with $S(0)=1-P(0)$, gives
\begin{equation}\label{eq:cuspminorant}
 \Theta(\alpha,L)-1\ge
 \alpha^{-2}P^\sharp(0)-P(0)+
 2e^{-(21\pi/20)\alpha}S((21\pi/10)\alpha).
\end{equation}

For $1\le\alpha\le9/4$, Lemma~\ref{lem:cuspdata} supplies such
polynomials, with the bound for $\Theta(\alpha,L)-1$ in
\eqref{eq:cuspminorant} exceeding
$(243/10)e^{-\pi\sqrt2\alpha}+1/4000$; Lemma~\ref{lem:reference}
then gives the uniform gap. For $\alpha\ge9/4$, the shortest-pair
contribution exceeds \eqref{eq:refbound}, proving \eqref{eq:cusp}.
\end{proof}

\begin{proposition}[Comparison at extreme scales]
\label{prop:extreme}
For every covolume-one lattice $L\subset\R^4$ and every
$\alpha\in(0,1/11]\cup[11,\infty)$,
\[
 \Theta(\alpha,L)\ge\Theta(\alpha,\D).
\]
At each such scale, equality holds precisely when $L$ is isometric to
$\D$.
\end{proposition}
\begin{proof}
For $\alpha\ge11$, Theorem~\ref{prop:Dcap} and
Proposition~\ref{prop:Acap} give the comparison on both trace
neighborhoods. Outside them, Corollary~\ref{cor:abound} gives
$a(L)\le25/\sqrt{351}<67/50$, so the shortest-pair contribution
exceeds \eqref{eq:refbound}. Equality therefore forces $L$ to be
isometric to $\D$. Poisson summation and isoduality give the
comparison and the same equality characterization for
$\alpha\le1/11$.
\end{proof}

Combining the preceding comparisons with the determinant bound
gives the following localization.
\begin{corollary}[Localization of the remaining pairs]
\label{prop:localization}
Let $L\subset\R^4$ have covolume one and not be isometric to $\D$,
let $\alpha\ge1$, and set $R=Q/a(L)$ for a Gram matrix $Q$ of $L$.
The strict comparison $\Theta(\alpha,L)>\Theta(\alpha,\D)$ holds
whenever at least one of the following conditions is satisfied:
\[
 \alpha\ge11,\quad a(L)\le\frac{21}{20},\quad
 h_P(R)\le\frac{11}{25}\text{ for some }P\in\DD,\quad
 h_P(R)\le\frac1{10}\text{ for some }P\in\Acl.
\]
For every remaining pair,
\begin{equation}\label{eq:prelocalization}
 1\le\alpha<11,\qquad
 \frac{21}{20}<a(L)\le\frac{25}{\sqrt{351}},
\end{equation}
and
\begin{equation}\label{eq:precapexclusion}
 h_P(R)>\frac{11}{25}\quad(P\in\DD),\qquad
 h_P(R)>\frac1{10}\quad(P\in\Acl).
\end{equation}
\end{corollary}
\begin{proof}
Theorem~\ref{prop:Dcap} and Proposition~\ref{prop:Acap} give the
strict comparison in the two trace neighborhoods for lattices not
isometric to $\D$. Propositions~\ref{prop:cusp} and~\ref{prop:extreme}
give it for short vectors and large scales. Outside these regions,
the strict trace bounds hold, and Corollary~\ref{cor:abound} gives
the upper bound for $a(L)$; the other parameter bounds follow from
the definitions of the regions.
\end{proof}

\section{Global comparison and proof of the theta theorem}
\label{sec:joint}\label{sec:certificate}\label{sec:finish}
In this section, we prove that the global minimum of $\Phi$ is
zero. Compactness gives attainment, and the following comparison
at stationary pairs determines the minimum and its equality cases.
\begin{theorem}[Comparison at stationary pairs]
\label{thm:stationary-comparison}
Let $L\subset\R^4$ have covolume one and let $\alpha\ge1$. If
\[
 \partial_\alpha\Phi(\alpha,L)=0,\qquad
 \left.\frac{d}{d\varepsilon}
 \Phi(\alpha,e^{\varepsilon H/2}L)\right|_{\varepsilon=0}=0
 \quad\text{for every }H=H^T\text{ with }\tr H=0,
\]
then
\[
 \Theta(\alpha,L)\ge\Theta(\alpha,\D),
\]
with equality if and only if $L$ is isometric to $\D$.
\end{theorem}
We call a pair stationary for $\Phi$ when these two derivative
conditions hold. The variations range over all positive scales and
covolume-one shapes.

The local comparisons cover the trace neighborhoods and short
vectors. At the remaining stationary pairs, the Gaussian moment
bound gives strict comparison for $\alpha\ge5$. Projection along
a shortest vector and a three-dimensional Fourier bound then give
positive lower bounds for the theta difference on the remaining
scalar rectangle.

\begin{lemma}[Attainment of the global minimum]
\label{lem:attainment}
The function $\Phi$ attains a global minimum over all covolume-one
lattice shapes and all positive scales. A minimizing pair can be
chosen with $\alpha\ge1$.
\end{lemma}
\begin{proof}
By Lemma~\ref{lem:sublevels} and continuity of theta, the set of
pairs $(\alpha,Q)$ with $Q$ Minkowski reduced, $\det Q=1$, and
\[
 \frac1{11}\le\alpha\le11,\qquad
 \Theta(\alpha,Q)\le\Theta(1/11,\D)
\]
is compact. It contains a representative of $(1,\D)$, so the
minimum of $\Phi$ on this set is attained and is at most zero.
For scales outside this interval, Proposition~\ref{prop:extreme}
gives $\Phi\ge0$. Inside the interval, pairs outside the set,
after reduction, satisfy
\[
 \Theta(\alpha,L)>\Theta(1/11,\D)\ge\Theta(\alpha,\D),
\]
so $\Phi>0$. Thus the attained minimum is global.
Equation~\eqref{eq:defectduality} permits a minimizing scale
$\alpha\ge1$.
\end{proof}

Reduction is used only to obtain compact representatives. The
minimum is global in the full space of positive scales and
covolume-one lattices, so all scale and volume-preserving shape
derivatives vanish, even at a boundary of the reduced domain or
when duality selects $\alpha=1$.

\begin{lemma}[First variations and the shortest-vector moment]
\label{lem:firstvariation}
Let $(\alpha,L)$ be stationary for $\Phi$, with $\alpha>0$, and put
$t=\pi\alpha$ and $a=a(L)$. Then
\begin{align}
 \sum_{v\in L}vv^Te^{-t|v|^2}&=\frac{M(t,L)}4I_4,
                  \label{eq:jointcov}\\
 \Theta(\alpha,L)-\Theta(\alpha,\D)
   &=t\bigl(M(t,L)-M(t,\D)\bigr),
                  \label{eq:jointscale}
\end{align}
and
\begin{equation}\label{eq:pairmoment}
 M(t,L)\ge8a e^{-ta}.
\end{equation}
If $\alpha\ge5$ and $21/20<a<67/50$, then
$\Theta(\alpha,L)>\Theta(\alpha,\D)$.
\end{lemma}
\begin{proof}
For symmetric $H$ with $\tr H=0$, the lattice
$L_\varepsilon=e^{\varepsilon H/2}L$ has covolume one. Since
$|e^{\varepsilon H/2}v|^2=v^Te^{\varepsilon H}v$,
Lemma~\ref{lem:sublevels} permits differentiation:
\[
 0=\left.\frac{d}{d\varepsilon}\Phi(\alpha,L_\varepsilon)
                                          \right|_{\varepsilon=0}
 =-\pi\alpha^2\sum_{v\in L}v^THv\,e^{-t|v|^2}.
\]
The symmetric matrices orthogonal to all traceless symmetric
matrices are precisely the scalar matrices. Taking the trace
therefore gives
\eqref{eq:jointcov}. Differentiation in scale gives
\[
 0=\partial_\alpha\Phi(\alpha,L)
 =\Theta(\alpha,L)-\Theta(\alpha,\D)
       -t\bigl(M(t,L)-M(t,\D)\bigr),
\]
which is \eqref{eq:jointscale}. For a shortest vector $v$, take
the quadratic form of \eqref{eq:jointcov} in the unit direction
$v/\sqrt a$. Retaining $v$ and $-v$ gives
$M(t,L)/4\ge2ae^{-ta}$, hence \eqref{eq:pairmoment}.

For $\alpha\ge5$ and $21/20<a<67/50$,
$8a e^{\pi\alpha(\sqrt2-a)}>34$. Thus \eqref{eq:MD34} and
\eqref{eq:pairmoment} give $M(t,L)>M(t,\D)$, and
\eqref{eq:jointscale} gives the strict theta comparison.
\end{proof}

The remaining scalar estimates concern the closed rectangle
\begin{equation}\label{eq:closedrectangle}
 \mathscr B=[1,5]\times[21/20,133441/100000].
\end{equation}
On this rectangle, $\tau=\pi\alpha a$ satisfies
\begin{equation}\label{eq:taurange}
 \frac{329}{100}<\tau<21.
\end{equation}

We next project along a shortest vector $v$. Set $a=|v|^2$ and define
\begin{equation}\label{eq:Gamma}
 \Gamma_a=a^{-1/2}\pi_{v^\perp}(L)\subset\R^3.
\end{equation}
Identify $v^\perp$ orthogonally with $\R^3$; radial sums are independent
of the choice.
\begin{lemma}[The projected lattice and its fibres]
\label{lem:fibres}
Let $v$ be a shortest vector of the covolume-one lattice $L$, set
$a=|v|^2$ and $e=v/\sqrt a$, and let $\Gamma_a$ be defined by
\eqref{eq:Gamma}. Then
\begin{equation}\label{eq:Gammafacts}
 \covol(\Gamma_a)=a^{-2},\qquad
 |y|^2\ge\frac34\quad(y\in\Gamma_a\setminus\{0\}).
\end{equation}
For each $y\in\Gamma_a\setminus\{0\}$, there is a signed displacement
$\widetilde\delta_y\in[-1/2,1/2]$ for which the fibre is
\[
 \left\{\sqrt a\bigl(y+(k+\widetilde\delta_y)e\bigr):k\in\Z\right\}.
\]
In scalar sums even in the parallel coordinate, one may replace
$\widetilde\delta_y$ by $\delta_y=|\widetilde\delta_y|$. Then
$0\le\delta_y\le1/2$ and
\begin{equation}\label{eq:fibreconstraint}
 w+\delta_y^2\ge1,\qquad w=|y|^2.
\end{equation}
The fibre over $0$ is $\Z v$.
\end{lemma}
\begin{proof}
A shortest vector is primitive, so $v$ extends to a basis of $L$
and $L\cap\R v=\Z v$. Projecting the other basis vectors gives
a basis of $\pi_{v^\perp}(L)$, whose covolume is
$\covol(L)/|v|=a^{-1/2}$; dilation by $a^{-1/2}$ gives
$\covol(\Gamma_a)=a^{-2}$. Each fibre is a translate of $\Z v$
with parallel shift $\sqrt a\,\widetilde\delta_y$ chosen so that
$|\widetilde\delta_y|\le1/2$. Substitution $k\mapsto-k$ makes even
scalar sums depend only on $\delta_y=|\widetilde\delta_y|$.
For $y\ne0$, minimality of $v$ gives $a(|y|^2+\delta_y^2)\ge a$,
hence \eqref{eq:fibreconstraint} and $|y|^2\ge3/4$.
\end{proof}

To express theta and its moments on each fibre, define, for $\tau>0$,
\begin{equation}\label{eq:fibresums}
 \vartheta_\tau(\delta)=\sum_{k\in\Z}e^{-\tau(k+\delta)^2},
 \qquad
 \varphi_\tau(\delta)=\sum_{k\in\Z}(k+\delta)^2e^{-\tau(k+\delta)^2}.
\end{equation}
Both functions are even and one-periodic, and their series converge
uniformly on $|\delta|\le1/2$ for fixed $\tau>0$.
The weighted fibre sums below are bounded by a constant multiple of
$(1+w)e^{-\tau w}$, summable over $\Gamma_a$; hence all regroupings
are justified by absolute convergence.

\begin{proposition}[The full-fibre stationarity identity]
\label{prop:fibreidentity}
Let $(\alpha,L)$ be a pair with $\alpha>0$ and $\covol(L)=1$ satisfying
\eqref{eq:jointcov} and \eqref{eq:jointscale}. Use the notation of
Lemma~\ref{lem:fibres}, and set $t=\pi\alpha$ and $\tau=ta$.
For arbitrary $\gamma,\eta\in\R$, define the zero-fibre contribution
and the general fibre weight by
\begin{align}
 Z_\tau={}&(1+\eta)\vartheta_\tau(0)
                  -(3\gamma+\eta\tau)\varphi_\tau(0),
                    \label{eq:Z}\\
 F_{\tau,\gamma,\eta}(w,\delta)={}&e^{-\tau w}
 \bigl[(1+\eta+(\gamma-\eta\tau)w)\vartheta_\tau(\delta)
             -(3\gamma+\eta\tau)\varphi_\tau(\delta)\bigr].
                    \label{eq:fibreF}
\end{align}
Then the following absolutely convergent identity holds:
\begin{equation}\label{eq:fibreidentity}
 \Theta(\alpha,L)=Z_\tau+
 \sum_{y\in\Gamma_a\setminus\{0\}}
 F_{\tau,\gamma,\eta}(|y|^2,\delta_y)
 -\eta\bigl(\Theta(\alpha,\D)-tM(t,\D)\bigr).
\end{equation}
\end{proposition}
\begin{proof}
Write $z=k+\delta_y$, with $\delta_0=0$ and $w=|y|^2$, using the
unsigned shifts permitted in scalar sums by Lemma~\ref{lem:fibres}. Set
\[
 S_{\parallel}=\sum_{y,k}z^2e^{-\tau(w+z^2)},\qquad
 S_{\perp}=\sum_{y,k}w e^{-\tau(w+z^2)}.
\]
The covariance trace and its component in direction $e$ give
\[
 a(S_{\perp}+S_{\parallel})=M(t,L),\qquad
 aS_{\parallel}=\frac{M(t,L)}4.
\]
Thus $S_\perp-3S_\parallel=0$, while the scale identity gives
\[
 \Theta(\alpha,L)-\tau(S_{\perp}+S_{\parallel})
                  =\Theta(\alpha,\D)-tM(t,\D).
\]
Consequently, summing over all lifts, including the origin,
\begin{align*}
 &\sum_{y\in\Gamma_a}\sum_{k\in\Z}
 \bigl[1+\eta+(\gamma-\eta\tau)w-(3\gamma+\eta\tau)z^2\bigr]
                                  e^{-\tau(w+z^2)}\\
 &\quad=\Theta(\alpha,L)+\gamma(S_{\perp}-3S_{\parallel})
                   +\eta\bigl[\Theta(\alpha,L)-\tau(S_{\perp}+S_{\parallel})\bigr]\\
 &\quad=\Theta(\alpha,L)+\eta\bigl(\Theta(\alpha,\D)-tM(t,\D)\bigr).
\end{align*}
Separating the zero fibre, $Z_\tau=F_{\tau,\gamma,\eta}(0,0)$,
proves \eqref{eq:fibreidentity}; both $y$ and $-y$ are included.
\end{proof}

For radial minorants of the nonzero fibres, we use the normalized
three-dimensional Laguerre basis
\begin{equation}\label{eq:Lag3}
 \ell_k(u)=\frac{L_k^{1/2}(u)}{L_k^{1/2}(0)}
 =\sum_{j=0}^k\frac{(-1)^j\binom{k}{j}}{(3/2)_j}u^j,
 \qquad (3/2)_0=1.
\end{equation}
Here $(3/2)_j=(3/2)(5/2)\cdots(j+1/2)$ for $j\ge1$.
If $P=\sum b_k\ell_k$ and $P^\sharp=\sum(-1)^kb_k\ell_k$, then
\begin{equation}\label{eq:FT3}
 \widehat{\bigl[e^{-\pi|y|^2}P(2\pi|y|^2)\bigr]}(\xi)
 =e^{-\pi|\xi|^2}P^\sharp(2\pi|\xi|^2).
\end{equation}
All projected auxiliary polynomials use this normalization.

\begin{lemma}[A projected Fourier lower bound]
\label{lem:projectedLP}
Let $(\alpha,L)$ be a pair with $\alpha>0$ and $\covol(L)=1$ satisfying the
stationarity identities \eqref{eq:jointcov}--\eqref{eq:jointscale}. Set
$a=a(L)$, $t=\pi\alpha$ and
$\tau=ta$. Let $P$ be a real polynomial in the basis \eqref{eq:Lag3},
with Fourier partner $P^\sharp$, and let $\gamma,\eta\in\R$.
Assume that $P^\sharp(u)\ge0$ for every $u\ge0$ and that
\begin{equation}\label{eq:potentialminorant}
 e^{-\tau}e^{-\pi w}P(2\pi w)
 \le F_{\tau,\gamma,\eta}(w,\delta)
 \quad\left(w\ge\tfrac34,\ 0\le\delta\le\tfrac12,
                      \ w+\delta^2\ge1\right).
\end{equation}
Then
\begin{align}\label{eq:Bexact}
 e^\tau(\Theta(\alpha,L)-\Theta(\alpha,\D))\ge
 \mathcal B(\alpha,a):={}&a^2P^\sharp(0)-P(0)+e^\tau Z_\tau\notag\\
 &-e^\tau\bigl[(1+\eta)\Theta(\alpha,\D)-\eta tM(t,\D)\bigr].
\end{align}
\end{lemma}
\begin{proof}
The Schwartz function
$f(y)=e^{-\tau}e^{-\pi|y|^2}P(2\pi|y|^2)$ has nonnegative Fourier
transform by \eqref{eq:FT3}, with $f(0)=e^{-\tau}P(0)$ and
$\widehat f(0)=e^{-\tau}P^\sharp(0)$.
Poisson summation, $\covol(\Gamma_a)=a^{-2}$, and
\eqref{eq:fibreidentity} give
\begin{align*}
 &e^\tau\bigl(\Theta(\alpha,L)-\Theta(\alpha,\D)\bigr)
       -\mathcal B(\alpha,a)\\
 &\quad=e^\tau\sum_{y\in\Gamma_a\setminus\{0\}}
       \bigl[F_{\tau,\gamma,\eta}(|y|^2,\delta_y)-f(y)\bigr]
       +a^2e^\tau\sum_{\xi\in\Gamma_a^*\setminus\{0\}}
          \widehat f(\xi)\ge0.
\end{align*}
Both sums are nonnegative by \eqref{eq:potentialminorant}
and Fourier positivity, respectively.
\end{proof}

We use finitely many auxiliary triples on closed intervals
$[\tau_-,\tau_+]$ with rational endpoints $\tau_-<\tau_+$,
parametrized by
\begin{equation}\label{eq:sigma}
 \sigma=\frac{2\tau-\tau_--\tau_+}{\tau_+-\tau_-}\in[-1,1].
\end{equation}
The functions used are affine in $\sigma$:
\begin{align}\label{eq:affine}
 P_\sigma(u)&=\sum_{k=0}^N(b_{0k}+\sigma b_{1k})\ell_k(u),
 &P_\sigma^\sharp(u)&=\sum_{k=0}^N(-1)^k(b_{0k}+\sigma b_{1k})\ell_k(u),\notag\\
 \gamma_\sigma&=g_0+\sigma g_1,
 &\eta_\sigma&=e_0+\sigma e_1.
\end{align}
Let $\mathscr W$ be the family of thirty auxiliary triples specified
in Appendix~\ref{app:coefficients}. Every derivative used in the
interval bounds includes the dependence of these coefficients on
$\tau=\pi\alpha a$.

\begin{lemma}[Uniform properties of the auxiliary family]
\label{lem:witnesses}
Let $(P_\sigma,\gamma_\sigma,\eta_\sigma)$ be any member of
$\mathscr W$, on its closed window $[\tau_-,\tau_+]$.
For every $\tau$ in that window, with $\sigma$ defined by
\eqref{eq:sigma},
\begin{align}
 P_\sigma^\sharp(u)&>0 &&(u\ge0),\label{eq:Qsign}\\
 P_\sigma(u)+(1+u)^2&<0 &&(u\ge60),\label{eq:Ptail}\\
 |\gamma_\sigma|&\le4,
 &|\eta_\sigma|&\le1,\label{eq:multipliers}
\end{align}
and \eqref{eq:potentialminorant} holds for every admissible
$(w,\delta)$. Each window is contained in $[329/100,21]$.
\end{lemma}
Appendix~\ref{subsec:potential-checks} proves Lemma~\ref{lem:witnesses}.
\begin{lemma}[A scalar lower-bound alternative]
\label{lem:finalcertificate}
For every $(\alpha,a)\in\mathscr B$, at least one of the following
inequalities holds:
\begin{align}
 8a&>e^{\pi\alpha a}M(\pi\alpha,\D),\label{eq:alternativeC}\\
 \mathcal B(\alpha,a)&>10^{-5}
 \quad\text{for an auxiliary triple in $\mathscr W$ valid at $\tau=\pi\alpha a$}.
                       \label{eq:alternativeB}
\end{align}
In the second alternative, $\mathcal B$ is formed from the indicated
auxiliary triple by \eqref{eq:Bexact}.
\end{lemma}
The proof in Appendix~\ref{subsec:scalar-checks} uses a finite
covering of $\mathscr B$ and the uniform truncation bounds of
Lemma~\ref{lem:scalartails}.

\begin{proof}[Proof of Theorem~\ref{thm:stationary-comparison}]
Theorem~\ref{prop:Dcap} and Proposition~\ref{prop:Acap} give the
assertion on the two trace neighborhoods. Outside them,
Corollary~\ref{cor:abound} gives $a(L)\le25/\sqrt{351}$, while
Proposition~\ref{prop:cusp} gives strict comparison if
$a(L)\le21/20$. For the remaining length range,
Lemma~\ref{lem:firstvariation} gives strict comparison at every
$\alpha\ge5$.

It remains to treat $1\le\alpha\le5$ and
$21/20<a=a(L)\le25/\sqrt{351}<133441/100000$.
Thus $(\alpha,a)\in\mathscr B$. Put $t=\pi\alpha$ and $\tau=ta$,
and apply Lemma~\ref{lem:finalcertificate}.
If \eqref{eq:alternativeC} holds, then
\eqref{eq:jointscale} and \eqref{eq:pairmoment} give
\[
 \Theta(\alpha,L)-\Theta(\alpha,\D)
 =t\bigl(M(t,L)-M(t,\D)\bigr)
 \ge t\bigl(8ae^{-\tau}-M(t,\D)\bigr)>0.
\]
If \eqref{eq:alternativeB} holds, the indicated auxiliary triple
satisfies the hypotheses of Lemma~\ref{lem:projectedLP} by
Lemma~\ref{lem:witnesses}. Hence
\[
 \Theta(\alpha,L)-\Theta(\alpha,\D)
 \ge e^{-\tau}\mathcal B(\alpha,a)>10^{-5}e^{-\tau}>0.
\]
The two alternatives cover $\mathscr B$. Equality occurs precisely
in the equality case of the local comparison at $\D$, namely when
$L$ is isometric to $\D$.
\end{proof}

\begin{proof}[Proof of Theorem~\ref{thm:main}]
Lemma~\ref{lem:attainment} gives a global minimizing pair with
$\alpha\ge1$. This pair is stationary, so
Theorem~\ref{thm:stationary-comparison} gives $\Phi(\alpha,L)\ge0$.
Consequently,
\[
 0\le\min_{\substack{\beta>0\\\covol(\Lambda)=1}}
       \Phi(\beta,\Lambda)\le\Phi(1,\D)=0.
\]
Thus $\Phi\ge0$ everywhere, which is \eqref{eq:main}.

Every pair satisfying equality is itself a global minimum.
For $\alpha\ge1$, stationarity and
Theorem~\ref{thm:stationary-comparison} give that $L$ is isometric
to $\D$. For $\alpha<1$, apply this conclusion to
$(\alpha^{-1},L^*)$ using \eqref{eq:defectduality}; isoduality of
$\D$ gives the same conclusion for $L$. Conversely, every
orthogonal image of $\D$ gives equality at every scale.
\end{proof}

\section{Epstein zeta functions and Gaussian mixtures}
\label{sec:epstein}\label{subsec:consequences}
In this section, we prove Theorem~\ref{thm:Epstein} by integrating
the theta comparison against the Mellin kernel. Positive
superposition gives the same conclusion for Gaussian mixtures in
Corollary~\ref{cor:mixtures}. Strictness at every scale yields
uniqueness in both cases.

\begin{proof}[Proof of Theorem~\ref{thm:Epstein}]
For $r>0$ and $s>0$, the Gamma integral gives
\begin{equation}\label{eq:Mellin-kernel}
 r^{-s}=\frac{\pi^s}{\Gamma(s)}
       \int_0^\infty\alpha^{s-1}e^{-\pi\alpha r}\,\dd\alpha.
\end{equation}
Summing over $L\setminus\{0\}$ and applying Tonelli's theorem gives,
for $s>2$,
\begin{equation}\label{eq:Mellin-Epstein}
 E(L,s)=\frac{\pi^s}{\Gamma(s)}
           \int_0^\infty\alpha^{s-1}\bigl(\Theta(\alpha,L)-1\bigr)\,\dd\alpha.
\end{equation}
For fixed $L$, exponential decay gives integrability at infinity.
At zero, Poisson summation gives
\[
 \Theta(\alpha,L)-1=\alpha^{-2}-1+
 O\!\left(\alpha^{-2}e^{-\pi a(L^*)/\alpha}\right),
 \qquad\alpha\downarrow0.
\]
Thus the integral converges for $s>2$.

Theorem~\ref{thm:main} makes the difference of the integrands
nonnegative, and strictly positive at every $\alpha>0$ unless
$L$ is isometric to $\D$. Integration proves
\eqref{eq:Epstein-main} and the equality statement.
\end{proof}

Write $\zeta_{\mathrm R}(s)=\sum_{n\ge1}n^{-s}$ for the Riemann zeta
function, $s>1$. For $s>2$, the shell expansion \eqref{eq:Dseries} gives
\begin{equation}\label{eq:reference-Epstein-product}
 E(\D,s)=24\,2^{-s/2}(1-2^{1-s})
                   \zeta_{\mathrm R}(s)\zeta_{\mathrm R}(s-1).
\end{equation}
Indeed, absolute convergence gives
\[
 \sum_{n\ge1}\frac{\sodd(n)}{n^s}
 =\zeta_{\mathrm R}(s)\sum_{\substack{d\ge1\\d\ \mathrm{odd}}}d^{1-s}
 =(1-2^{1-s})\zeta_{\mathrm R}(s)\zeta_{\mathrm R}(s-1).
\]

Monotone convergence extends the comparison to positive Gaussian mixtures.
\begin{corollary}[Positive Gaussian mixtures]
\label{cor:mixtures}
Let $\nu$ be a nonzero positive Borel measure on $(0,\infty)$, and set
\[
 f(r)=\int_{(0,\infty)}e^{-\pi\alpha r}\,\dd\nu(\alpha),\qquad r>0.
\]
Assume that $\sum_{v\in\D\setminus\{0\}}f(|v|^2)<\infty$.
For every covolume-one lattice $L\subset\R^4$,
\[
 \sum_{v\in L\setminus\{0\}}f(|v|^2)
 \ge\sum_{v\in\D\setminus\{0\}}f(|v|^2),
\]
where the left-hand side may be infinite. Equality holds precisely
when $L$ is isometric to $\D$.
\end{corollary}
\begin{proof}
Monotone convergence of the nonnegative finite lattice sums gives
\[
 \sum_{v\in L\setminus\{0\}}f(|v|^2)
   =\int_{(0,\infty)}\bigl(\Theta(\alpha,L)-1\bigr)\,\dd\nu(\alpha),
\]
with both sides possibly infinite. Theorem~\ref{thm:main} gives the
comparison. If $L$ is not isometric to $\D$ and has finite energy,
subtraction gives the integral of a strictly positive function
against the nonzero positive measure $\nu$. If its energy is
infinite, strictness follows from the finite reference energy.
\end{proof}

\appendix
\section{Proofs of the auxiliary estimates}
\label{sec:verification}\label{app:checks}
This appendix proves the auxiliary estimates used in the local and
global comparisons. The coefficients are specified in
Appendix~\ref{app:coefficients}; finite interval bounds and uniform
tail estimates establish the required inequalities.
\begin{theorem}[Uniform auxiliary estimates]
\label{prop:auxiliary-estimates}
The reference bounds \eqref{eq:refbound}--\eqref{eq:MD34}, the moment
estimates \eqref{eq:highband} and \eqref{eq:rootcert}, and the
short-vector estimate \eqref{eq:cusp-finite} hold on their respective
domains. The rational affine family $\mathscr W$ of Appendix~\ref{app:coefficients}
satisfies all the assertions of Lemma~\ref{lem:witnesses} and gives
the alternative in Lemma~\ref{lem:finalcertificate} at every point
of the closed rectangle $\mathscr B$.
\end{theorem}
We reduce the assertions to exact polynomial signs and finite-sum
inequalities on closed boxes. The local estimate near the identity
handles the vanishing denominators in the shell bounds; uniform
tail estimates extend the finite-sum comparisons to the full series.
Lemma~\ref{lem:audit-cover} gives the covering argument for the
projected minorants and the scalar alternative.
Throughout, $t=\pi\alpha$ and $\tau=ta$; when a lattice is fixed,
$a=a(L)$.

\subsection{Interval bounds and polynomial signs}\label{subsec:arithmetic}\label{app:data}
Every prescribed initial box is included. Rational midpoint
bisection retains both closed children, and the complete subdivision
is checked. A shell box is accepted only by strict infeasibility,
applicability of the local lemma throughout its feasible part, or
a uniform interval bound. For an auxiliary triple, its closed window
must contain the entire $\pi\alpha a$-range of the assigned box.

Exact input checks match the printed coefficients with the
verification data. The paper specifies the domains, terminal
inequalities, and analytic error bounds; the embedded archive
contains the complete subdivisions, programs, and execution records.

For a closed box $B=\prod_i[c_i-r_i,c_i+r_i]$ and a function $g$
of class $C^1$ on a neighborhood of $B$, the mean-value estimate gives
\begin{equation}\label{eq:meanvaluecheck}
 g(x)\ge g(c)-\sum_i r_i\sup_B|\partial_i g|,\qquad x\in B.
\end{equation}
If $g$ is of class $C^2$, Taylor's theorem also gives
\begin{equation}\label{eq:Taylorcheck}
 g(x)\ge g(c)-\sum_i r_i|\partial_i g(c)|
 -\frac12\sum_{i,j}r_ir_j\sup_B|\partial_i\partial_jg|,
 \qquad x\in B.
\end{equation}
The double sum is over ordered pairs. A derivative of certified
fixed sign on the current box reduces the minimum to a coordinate
face. Outward rounding bounds $g(c)$ below and the absolute
derivatives above; every divisor interval excludes zero.
Derivatives include all parameter dependence of the finite
expressions. The analytic truncation errors are subtracted afterward.

Exact Sturm sequences establish the endpoint signs in
Lemmas~\ref{lem:short-vector-endpoints} and~\ref{lem:projected-endpoints}.
Affine interpolation preserves these signs and the multiplier bounds.
Exponential enclosures use
\[
 \left|e^y-\sum_{j=0}^{24}\frac{y^j}{j!}\right|
 \le\frac{2}{16^{25}25!},\qquad |y|\le\frac1{16},
\]
with reduction by powers of two. Machin's identity and alternating
arctangent series enclose $\pi$; integer arithmetic verifies rational
bounds for square and fourth roots.

The finite inequalities are verified with $75$-digit decimal
intervals, using both correctly rounded elementary-function
enclosures and the explicit Taylor bound above. A separate
verification uses exact rational polynomial arithmetic and
$192$-bit MPFR intervals with directed rounding~\cite{MPFR}.
Both implementations cover the same domains with the same
acceptance conditions and truncation margins; their subdivisions
may differ.

\subsection{Reference, short-vector, and shell estimates}\label{subsec:reference-checks}
\label{subsec:spectral-checks}\label{subsec:root-checks}
\begin{proof}[Proof of Lemma~\ref{lem:reference}]
Put $q=e^{-\pi\sqrt2\alpha}$. For $\alpha\ge1$, $q<1/85$;
the first six shells and $\sodd(n)\le n^2\le49\,2^{n-7}$ for
$n\ge7$ give
\[
 e^{\pi\sqrt2\alpha}\bigl(\Theta(\alpha,\D)-1\bigr)
 \le24\left(1+q+4q^2+q^3+6q^4+4q^5+
                   \frac{49q^6}{1-2q}\right)<\frac{243}{10}.
\]
For $\alpha\ge5$, $q<10^{-4}$ gives
\[
 e^{\pi\sqrt2\alpha}M(\pi\alpha,\D)
 \le24\sqrt2\sum_{n\ge1}n^3q^{n-1}
 =24\sqrt2\frac{1+4q+q^2}{(1-q)^4}<34.
\]
These are \eqref{eq:refbound} and \eqref{eq:MD34}.
\end{proof}

Proposition~\ref{prop:cusp} uses the following polynomial estimate.
\begin{lemma}[Short-vector auxiliary polynomials]
\label{lem:cuspdata}
There is a piecewise-affine family
$\alpha\mapsto P_\alpha$, $\alpha\in[1,9/4]$, of real polynomials
of degree at most nine, expressed in the basis $\ell_k^{(4)}$,
such that $S_\alpha=1-P_\alpha$ satisfies
\[
 \begin{aligned}
 S_\alpha(u)&>0,\quad P_\alpha^\sharp(u)>0 &&(u\ge0),\\
 S_\alpha(u)-2S_\alpha'(u)&>0 &&(0\le u\le(21\pi/10)\alpha).
 \end{aligned}
\]
For this family,
\begin{multline}\label{eq:cusp-finite}
 \alpha^{-2}P_\alpha^\sharp(0)-P_\alpha(0)
 +2e^{-(21\pi/20)\alpha}S_\alpha((21\pi/10)\alpha)
 -\frac{243}{10}e^{-\pi\sqrt2\alpha}>\frac1{4000},\\
 1\le\alpha\le9/4.
\end{multline}
\end{lemma}
\begin{proof}
Use the twelve rational polynomials of
Lemma~\ref{lem:short-vector-endpoints}, at the scale nodes
\begin{equation}\label{eq:cusp-nodes}
 1,\quad17/16,\quad9/8,10/8,\ldots,18/8.
\end{equation}
Linear interpolation preserves the signs, since
Lemma~\ref{lem:short-vector-endpoints} gives them on a common interval
containing $[0,(21\pi/10)\alpha]$ throughout each cell. The supplied
rational intervals have consecutive endpoints and cover every
interpolation cell. On each interval, direct enclosure or
\eqref{eq:meanvaluecheck} proves \eqref{eq:cusp-finite}, including
both endpoints.
\end{proof}

\begin{proof}[Proof of Lemma~\ref{lem:highband}]
The case $u=\max_i\lambda_i=1$ is the identity. Fix
$1<u\le5/3$ with nonempty constraint set and minimize
$\Hh_{44/5}(m,V,\eta)-1-m/10$ over three eigenvalues in $[2/3,u]$
with product $1/u$. Then $\eta=u-\overline\lambda>0$, since equality
forces every eigenvalue to equal $u$. Compactness and smoothness give
a minimum.

As $u\downarrow1$, the product constraint gives
$u^{-3}\le\lambda_i\le u$, so a neighborhood of the identity is
covered by Lemma~\ref{lem:smallband} and convexity in $s$.
On the remaining terminal boxes, positivity of $\eta$ is verified
before division.
At fixed $m,u$, writing $\eta=u-m-1$,
\begin{equation}\label{eq:HV}
 \frac{\partial\Hh_s}{\partial V}
 =\frac{e^{-s(m+\eta)}\eta^2}{(\eta^2+V)^2}
       \bigl(e^y(y-1)+1\bigr)>0,
 \qquad y=s(\eta+V/\eta)>0.
\end{equation}
The positivity follows by differentiating $e^y(y-1)+1$.
For the constrained minimum, write
\[
 \mathcal F(m,V;u)=\Hh_s(m,V,u-m-1)-1-\frac{m}{10},
 \qquad s=44/5.
\]
Here $\mathcal F_m$ includes the dependence of $\eta=u-m-1$ on $m$,
and $\mathcal F_V=\partial_V\Hh_s>0$. Since $\partial_i m=1/4$
and $\partial_i V=(\lambda_i-\overline\lambda)/6$, the constraint
$\sum_{i=2}^4\log\lambda_i=-\log u$ gives the multiplier equation
\[
 \frac{\mathcal F_V}{6}\lambda_i^2+
 \left(\frac{\mathcal F_m}{4}
       -\frac{\mathcal F_V\overline\lambda}{6}\right)\lambda_i
 =\nu,\qquad i=2,3,4.
\]
The constraint gradient is nonzero; at an interior minimum, two free
eigenvalues agree because all three solve a nonzero quadratic.
At a boundary minimum, one is $2/3$ or $u$. Thus, for every fixed
maximum $u$, the minimum is attained on one of the following
surfaces, up to permutation:
\begin{equation}\label{eq:highsurfaces}
 (u,x,x,(ux^2)^{-1}),\quad
 (u,2/3,x,3/(2ux)),\quad
 (u,u,x,(u^2x)^{-1}).
\end{equation}
Parameterize each surface by $u=1+2r/3$ and $x=2/3+t$,
with $(r,t)\in[0,1]^2$. The full-square subdivision includes exact
interval exclusions outside the spectral constraints $\lambda_i\in[2/3,u]$.

For $m,V,\eta>0$, multiplication by the positive denominator converts
\eqref{eq:highband} to
\begin{equation}\label{eq:Kspectral}
 K_s=\eta^2(e^{sV/\eta}-1)+V(e^{-s\eta}-1)
 -(\eta^2+V)(e^{sm}-1)
 -\frac{m}{10}(\eta^2+V)e^{sm}\ge0,
 \quad s=44/5.
\end{equation}
On each terminal box of the supplied subdivisions of $[0,1]^2$,
the recorded bounds establish a strict spectral exclusion, the
hypotheses of Lemma~\ref{lem:smallband}, or a positive mean-value
lower bound for \eqref{eq:Kspectral}. In the second case, convexity
in $s$ gives the required estimate at $s=44/5$. The closed boxes
cover the feasible parts of all three surfaces. Since each
constrained minimum occurs there, the estimate holds throughout
the spectral band.
\end{proof}

For the root shell, use $\delta_v\le m+\eta$ from \eqref{eq:rootstats}.
\begin{proof}[Proof of Lemma~\ref{lem:rootcert}]
Order the three eigenvalues other than $u$ as $x\le y\le z$ and set
$z=4-u-x-y$. There are two initial boxes,
\[
 [1,48/37]\times[26/37,166/111]^2,
 \qquad [48/37,166/111]\times[26/37,166/111]^2.
\]
A box is discarded only if it strictly violates the eigenvalue, ordering, or product constraints. If all $\mu w_i$ belong to $[9/10,11/10]$, then $u<48/37$ and
$m+\eta=\mu u-1\le1/10$, so Lemma~\ref{lem:smallband} applies.

On remaining boxes, set $\eta=\mu(u-1)$ in the first domain and
$\eta=11\mu/37$ in the second; these smooth branches agree at $u=48/37$.
Check $w_i>0$ and $\eta>0$ throughout each box, then bound
\eqref{eq:Kspectral} below by the mean-value estimate, with $s=22/5$
and $m/200$ replacing $m/10$. The bound is strictly positive.
The identity case $u=1$ uses the local estimate, avoiding division by zero.

The supplied rational subdivisions cover both initial boxes, and
the recorded bounds verify one of these conditions on every terminal
box. The covering argument therefore proves \eqref{eq:rootcert}
throughout \eqref{eq:wrelax}.
\end{proof}

\subsection{Projected minorants and the scalar alternative}\label{subsec:potential-checks}
\label{subsec:scalar-checks}\label{app:derivatives}
For the triples of Lemma~\ref{lem:projected-endpoints}, we prove
the fibre minorization and the scalar alternative.

\begin{proof}[Proof of Lemma~\ref{lem:witnesses}]
Lemma~\ref{lem:projected-endpoints} gives the signs and bounds at
$\sigma=\pm1$; the polynomials' and multipliers' affine dependence
extends them throughout each window.

Multiplying \eqref{eq:potentialminorant} by $e^{\tau+\pi w}$ reduces
the comparison to a lower bound for
\begin{equation}\label{eq:Gexact}
 G=e^{\tau+(\pi-\tau)w}
 \bigl[A\vartheta_\tau(\delta)-C\varphi_\tau(\delta)\bigr]
 -P_\sigma(2\pi w),
\end{equation}
where $A=1+\eta+(\gamma-\eta\tau)w$ and
$C=3\gamma+\eta\tau$. Parameterize the complete shortest-vector
constraint by
\begin{equation}\label{eq:potentialbox}
 -1\le\sigma\le1,\quad0\le\delta\le\tfrac12,
 \quad0\le x\le\tfrac{37}{4},\qquad w=1-\delta^2+x.
\end{equation}
This contains all admissible pairs with $w\le10$ and some with
$10<w\le41/4$. Retaining $-5\le k\le5$ in both sums defines $G_5$.
Uniformly for $\tau\ge329/100$ and admissible shifts,
\begin{equation}\label{eq:fibretail}
 \sum_{|k|\ge6}e^{-\tau(k+\delta)^2}
 \le4e^{-(329/100) 121/4},\qquad
 \sum_{|k|\ge6}(k+\delta)^2e^{-\tau(k+\delta)^2}
 \le121e^{-(329/100) 121/4}.
\end{equation}
On each omitted half-fibre, the absolute coordinate starts at
$z_0\ge11/2$ and increases by one. Both summands decrease, with
ratios at most $(13/11)^2e^{-12(329/100)}<1/2$; twice the first
term gives \eqref{eq:fibretail}.
On \eqref{eq:potentialbox}, $|A|<259$, $|C|\le33$,
and $\tau+(\pi-\tau)w<8$; hence \eqref{eq:fibretail} gives
\begin{equation}\label{eq:Gtail}
 |G-G_5|\le
 e^8(4\cdot259+121\cdot33)e^{-(329/100)121/4}<10^{-30}.
\end{equation}
It therefore suffices to verify
\begin{equation}\label{eq:Gcertificate}
 G_5>10^{-30}\quad\text{throughout the domain \eqref{eq:potentialbox}}.
\end{equation}
For each of the thirty triples, the supplied rational subdivision
covers \eqref{eq:potentialbox}. On every terminal box, direct
enclosure, \eqref{eq:Taylorcheck}, or reduction to a minimizing face
gives \eqref{eq:Gcertificate}. The tail bound \eqref{eq:Gtail} then
yields $G>0$ on every admissible fibre with $w\le10$.

For $w\ge10$, Poisson summation and a geometric series give
$\vartheta_b(\delta)\le\vartheta_b(0)<2$ for $b\ge3/2$.
With $z^2e^{-\tau z^2}\le[2/(e\tau)]e^{-\tau z^2/2}$,
we obtain $\vartheta_\tau(\delta)<2$ and $\varphi_\tau(\delta)<1$
for admissible $\tau$, hence
\begin{equation}\label{eq:infiniteradial}
 \left|e^{\tau+(\pi-\tau)w}(A\vartheta_\tau-C\varphi_\tau)\right|
 \le e^{\tau+(\pi-\tau)w}[16+\tau+(8+2\tau)w]
 <9(20+15w).
\end{equation}
For $w\ge10$, the middle expression decreases with $\tau$;
evaluation at $\tau=329/100$ gives the last bound.
Since $2\pi w>60$ and $(1+2\pi w)^2>9(20+15w)$ for $w\ge10$,
\eqref{eq:Ptail} gives $G>0$ on the remaining radial range.
\end{proof}

For the finite derivatives, use
\[
 m_r(\tau,\delta)=\sum_{k=-5}^{5}(k+\delta)^r
 e^{-\tau(k+\delta)^2},\qquad
 \partial_\tau m_r=-m_{r+2},\qquad
 \partial_\delta m_r=r m_{r-1}-2\tau m_{r+1}.
\]
Here $r$ is a nonnegative integer, with the first term omitted at $r=0$.
For $A,C$ in \eqref{eq:Gexact},
\[
 G_5=e^{\tau+(\pi-\tau)w}(A m_0-Cm_2)-P_\sigma(2\pi w).
\]
The recurrences through $m_6$ and the product and chain rules give
all first and second derivatives, including the $\tau$-dependence
of $P_\sigma$, $\gamma_\sigma$, and $\eta_\sigma$, where
$d\sigma/d\tau=2/(\tau_+-\tau_-)$.

For the coordinate change, put $h=(\tau_+-\tau_-)/2$ and
\[
 \chi(\sigma,\delta,x)
 =\bigl((\tau_-+\tau_+)/2+h\sigma,\delta,1-\delta^2+x\bigr),
 \qquad
 J_\chi=\begin{pmatrix}h&0&0\\0&1&0\\0&-2\delta&1\end{pmatrix}.
\]
Writing $\widetilde G_5=G_5\circ\chi$ and $e_2=(0,1,0)^T$, we have
\[
 \nabla\widetilde G_5=J_\chi^T\nabla G_5,\qquad
 D^2\widetilde G_5=J_\chi^T(D^2G_5)J_\chi
                     -2(G_5)_w e_2e_2^T,
\]
with derivatives of $G_5$ evaluated at $\chi(\sigma,\delta,x)$.
Use these in \eqref{eq:Taylorcheck}, enclosing polynomials and their
derivatives by exact finite Taylor expansions at rational midpoints.

Fix a triple whose window contains $\tau=\pi\alpha a$.
In \eqref{eq:Bexact}, the $k=0$ term of $Z_\tau$ cancels the
constant term of $(1+\eta_\sigma)\Theta(\alpha,\D)$. Truncating to
$1\le k\le5$ in the zero-fibre sum and $1\le n\le20$ in the
reference series gives
\begin{align}\label{eq:Bfinite}
 \mathcal B_{5,20}(\alpha,a)={}&a^2P_\sigma^\sharp(0)-P_\sigma(0)\notag\\
 &+2\sum_{k=1}^5[1+\eta_\sigma-(3\gamma_\sigma+
                      \eta_\sigma\tau)k^2]e^{\tau(1-k^2)}\notag\\
 &+24\sum_{n=1}^{20}\sodd(n)
       [\eta_\sigma t\sqrt2\,n-1-\eta_\sigma]e^{\tau-t\sqrt2\,n},
       \qquad t=\pi\alpha.
\end{align}
\begin{lemma}[Uniform truncation bounds]
\label{lem:scalartails}
Fix $(\alpha,a)\in\mathscr B$ and an auxiliary triple whose window
contains $\tau=\pi\alpha a$. The exact lower bound and its finite
approximation in \eqref{eq:Bfinite} satisfy
\begin{equation}\label{eq:scalartail}
 |\mathcal B(\alpha,a)-\mathcal B_{5,20}(\alpha,a)|<10^{-24}.
\end{equation}
The estimate is uniform over all such choices. Independently of the
choice of an auxiliary triple, for every $(\alpha,a)\in\mathscr B$
the positive remainder after $n=20$ in
\[
 e^\tau M(t,\D)=24\sum_{n\ge1}\sodd(n)\sqrt2\,n\,e^{\tau-t\sqrt2\,n}
\]
is smaller than $10^{-24}$, where $t=\pi\alpha$ and $\tau=ta$.
\end{lemma}
\begin{proof}
Every auxiliary triple satisfies \eqref{eq:multipliers}, and
$329/100<\tau<21$ on \eqref{eq:closedrectangle}.
The omitted zero-fibre terms have absolute sum at most
\[
 2\sum_{k\ge6}(2+33k^2)e^{-\tau(k^2-1)}
 \le4760e^{-(329/100)35}<10^{-30}.
\]
The consecutive-term ratio is at most
$(7/6)^2e^{-13(329/100)}<1/2$, so twice the first term bounds the
series.
For the reference tail,
$|\eta t\sqrt2\,n-1-\eta|\le26n$ since $t<16$ and $\sqrt2<3/2$.
Using $\sodd(n)\le n^2$, $a<27/20$, $\sqrt2>7/5$, and $t>3$, gives
an absolute term bound
\[
 624n^3\exp\left[-3\left(\frac75n-\frac{27}{20}\right)\right].
\]
The ratio is at most $(22/21)^3e^{-21/5}<1/2$ and the first exponent
is less than $-84$, giving $1248\cdot21^3e^{-84}$ for the tail and less than
$10^{-24}$ for both tails together. Since
$24\sqrt2\,n\sodd(n)\le36n^3<624n^3$, this majorant also bounds
the positive reference-moment tail on $\mathscr B$, without a window
condition. The bound for $\mathcal B-\mathcal B_{5,20}$ is uniform
over valid triples.
\end{proof}

The preceding estimates give the following finite covering criterion.
For a finite expression $q$ on a closed box $B$,
$\underline q(B)$ and $\overline q(B)$ bound $q$ below and above
on $B$, including rounding errors.

\begin{lemma}[Finite covering criterion]
\label{lem:audit-cover}
Consider the auxiliary functions and initial boxes specified in this
appendix. Suppose that the following conditions hold.
\begin{enumerate}
\item Each auxiliary window is a closed interval contained in
$[329/100,21]$. At its two endpoints $\sigma=\pm1$,
\[
 P_\sigma^\sharp(u)>0\quad(u\ge0),\qquad
 P_\sigma(u)+(1+u)^2<0\quad(u\ge60),
 \qquad |\gamma_\sigma|\le4,\quad|\eta_\sigma|\le1.
\]
\item For each auxiliary triple, a finite subdivision of
$[-1,1]\times[0,1/2]\times[0,37/4]$, with
$\tau=((1-\sigma)\tau_-+(1+\sigma)\tau_+)/2$ and $w=1-\delta^2+x$,
certifies $G_5>10^{-30}$ on every terminal box by interval
enclosure, Taylor's theorem, or monotonicity.
\item A finite subdivision of $\mathscr B$ has, on each terminal box
$B$, either
\begin{equation}\label{eq:audit-C-leaf}
 \underline{\,8a-\mathfrak m_{20}(\alpha,a)\,}(B)>10^{-24},
 \qquad
 \mathfrak m_{20}=24\sum_{n=1}^{20}\sodd(n)\sqrt2\,n\,
                          e^{\pi\alpha(a-\sqrt2\,n)},
\end{equation}
or, for an auxiliary triple whose window contains the whole
$\tau$-range of $B$,
\begin{equation}\label{eq:audit-E-leaf}
 \underline{\mathcal B_{5,20}}(B)>10^{-5}+10^{-24}.
\end{equation}
\item Each subdivision is obtained by splitting a closed rational
box into two closed rational boxes whose union is the parent box.
Every branch terminates at a box on which the indicated inequality
is certified, and the verification consumes the complete subdivision
record.
\end{enumerate}
Then Lemmas~\ref{lem:witnesses} and~\ref{lem:finalcertificate} hold,
with energy margin $10^{-5}$ in the latter.
\end{lemma}
\begin{proof}
Affine interpolation preserves the signs and multiplier bounds.
By \eqref{eq:Gtail}, the terminal bound gives $G>0$ on
\eqref{eq:potentialbox}, hence for all admissible fibres with $w\le10$;
\eqref{eq:infiniteradial} and \eqref{eq:Ptail} treat $w\ge10$.

Lemma~\ref{lem:scalartails} bounds both scalar tails by $10^{-24}$,
so \eqref{eq:audit-C-leaf} gives \eqref{eq:alternativeC} and
\eqref{eq:audit-E-leaf} gives \eqref{eq:alternativeB}.
The closed children cover their parent interval; induction over the finite trees ensures that every point of the initial domains, including their boundaries, is covered.
\end{proof}

\begin{proof}[Proof of Lemma~\ref{lem:finalcertificate}]
In the rational subdivision of \eqref{eq:closedrectangle}, each terminal
rectangle uses either the covariance alternative or an auxiliary index.
In the first case, on $B$,
\[
 \overline{\mathfrak m_{20}}(B)+10^{-24}<\underline{8a}(B).
\]
Lemma~\ref{lem:scalartails} gives
$e^\tau M(t,\D)<8a$ throughout $B$, proving
\eqref{eq:alternativeC} there.

In the second case, $B$ has index $j$ with $\tau(B)\subset I_j$,
and enclosure or \eqref{eq:meanvaluecheck} gives
\begin{equation}\label{eq:finalpredicate}
 \mathcal B_{5,20}>10^{-5}+10^{-24}
\end{equation}
throughout that rectangle. The chain rule includes
$\partial_\alpha\tau=\pi a$ and $\partial_a\tau=\pi\alpha$, together
with the affine dependence of $P,\gamma,\eta$ on $\tau$.
Lemma~\ref{lem:scalartails} then gives
\eqref{eq:alternativeB}.

The supplied subdivision verifies the indicated inequality on
every terminal rectangle, together with the window inclusion
whenever an auxiliary triple is used. Lemma~\ref{lem:audit-cover}
therefore gives the alternative on all of $\mathscr B$, including
its boundary.
\end{proof}

\begin{proof}[Proof of Theorem~\ref{prop:auxiliary-estimates}]
Appendix~\ref{subsec:reference-checks} proves the reference and shell
estimates and Lemma~\ref{lem:cuspdata}; Lemmas~\ref{lem:witnesses}
and~\ref{lem:finalcertificate} follow from Appendix~\ref{subsec:potential-checks}.
\end{proof}

\section{Rational auxiliary functions}\label{app:coefficients}
This appendix specifies the rational auxiliary functions used in
Appendix~\ref{sec:verification}. We give the nodes, closed windows,
degrees, multipliers, and integer coefficient vectors, with one
denominator for each family. Entries are read from left to right
across successive lines; affine interpolation extends the endpoint
estimates to each interval.

\subsection{Polynomials for the short-vector bound}\label{subsec:coefficient-cusp}
At each of the twelve nodes below, define
\[
 P_\alpha(u)=10^{-16}\sum_{k=0}^{9}a_{\alpha,k}\ell_k^{(4)}(u),
 \qquad \ell_k^{(4)}(u)=\frac{L_k^1(u)}{k+1},
 \qquad \mathbf a_\alpha=(a_{\alpha,0},\ldots,a_{\alpha,9}).
\]
Interpolate linearly in $\alpha$ between nodes. Multiplying the $k$th
coefficient by $(-1)^k$ gives the Fourier partner; set $S_\alpha=1-P_\alpha$.

\begin{lemma}[Short-vector endpoint polynomials]
\label{lem:short-vector-endpoints}
Let $P_\alpha$ be the twelve degree-nine polynomials specified at the
nodes \eqref{eq:cusp-nodes} by the integer vectors below. Their Laguerre coefficients belong to $10^{-16}\mathbb Z$.
For consecutive nodes $a<b$ and $\alpha\in\{a,b\}$,
\[
 \begin{aligned}
 S_\alpha(u)&>0,\qquad P_\alpha^\sharp(u)>0 &&(u\ge0),\\
 S_\alpha(u)-2S_\alpha'(u)&>0 &&(0\le u\le33b/5).
 \end{aligned}
\]
\end{lemma}
\begin{proof}
Exact rational Sturm sequences give no roots of $S_\alpha$ or
$P_\alpha^\sharp$ on $[0,\infty)$; both are positive at zero.
Likewise, $S_\alpha-2S_\alpha'$ has no roots on $[0,33b/5]$ and
positive endpoint values, using the larger adjacent interval at shared nodes.
\end{proof}

\begin{samepage}
\footnotesize
\setlength{\abovedisplayskip}{5pt}
\setlength{\belowdisplayskip}{5pt}
\noindent $\alpha=1$.
\[
\begin{aligned}
\mathbf a_{1}={}\bigl( & 4\,586\,315\,448\,202\,439,\; -1\,575\,758\,594\,260\,412,\; 53\,663\,846\,908\,617,\; 219\,065\,131\,076\,022,\\
& -21\,544\,209\,473\,589,\; -38\,092\,436\,935\,344,\; 6\,091\,221\,656\,916,\; 11\,121\,805\,878\,183,\\
& -2\,657\,442\,059\,217,\; 177\,239\,424\,585\bigr)
\end{aligned}
\]
\end{samepage}
\normalsize
\par\penalty0\smallskip

\begin{samepage}
\footnotesize
\setlength{\abovedisplayskip}{5pt}
\setlength{\belowdisplayskip}{5pt}
\noindent $\alpha=\frac{17}{16}$.
\[
\begin{aligned}
\mathbf a_{\frac{17}{16}}={}\bigl( & 3\,888\,262\,756\,862\,798,\; -1\,520\,158\,466\,539\,241,\; 139\,542\,442\,260\,210,\; 199\,039\,692\,190\,185,\\
& -56\,619\,662\,427\,504,\; -28\,276\,899\,637\,239,\; 15\,399\,087\,507\,243,\; 6\,822\,342\,167\,553,\\
& -6\,104\,441\,184\,831,\; 1\,287\,586\,769\,446\bigr)
\end{aligned}
\]
\end{samepage}
\normalsize
\par\penalty0\smallskip

\begin{samepage}
\footnotesize
\setlength{\abovedisplayskip}{5pt}
\setlength{\belowdisplayskip}{5pt}
\noindent $\alpha=\frac{9}{8}$.
\[
\begin{aligned}
\mathbf a_{\frac{9}{8}}={}\bigl( & 3\,279\,211\,963\,693\,232,\; -1\,428\,335\,508\,453\,371,\; 208\,201\,285\,759\,203,\; 167\,478\,243\,420\,708,\\
& -77\,929\,147\,715\,859,\; -14\,495\,679\,797\,049,\; 19\,399\,977\,548\,226,\; 1\,562\,218\,452\,531,\\
& -6\,395\,452\,950\,744,\; 2\,301\,888\,009\,178\bigr)
\end{aligned}
\]
\end{samepage}
\normalsize
\par\penalty0\smallskip

\begin{samepage}
\footnotesize
\setlength{\abovedisplayskip}{5pt}
\setlength{\belowdisplayskip}{5pt}
\noindent $\alpha=\frac{5}{4}$.
\[
\begin{aligned}
\mathbf a_{\frac{5}{4}}={}\bigl( & 2\,271\,448\,239\,644\,282,\; -1\,180\,918\,043\,924\,246,\; 290\,795\,525\,739\,180,\; 91\,296\,705\,276\,414,\\
& -88\,888\,458\,855\,141,\; 11\,845\,698\,901\,641,\; 16\,325\,489\,947\,734,\; -5\,837\,553\,576\,261,\\
& -3\,167\,567\,261\,595,\; 2\,843\,839\,547\,604\bigr)
\end{aligned}
\]
\end{samepage}
\normalsize
\par\penalty0\smallskip

\begin{samepage}
\footnotesize
\setlength{\abovedisplayskip}{5pt}
\setlength{\belowdisplayskip}{5pt}
\noindent $\alpha=\frac{11}{8}$.
\[
\begin{aligned}
\mathbf a_{\frac{11}{8}}={}\bigl( & 1\,524\,616\,508\,136\,497,\; -910\,452\,688\,024\,760,\; 305\,137\,293\,199\,002,\; 23\,000\,755\,954\,374,\\
& -73\,358\,028\,043\,542,\; 27\,924\,372\,033\,525,\; 6\,310\,520\,424\,846,\; -7\,374\,353\,230\,926,\\
& 309\,019\,264\,983,\; 2\,084\,079\,421\,215\bigr)
\end{aligned}
\]
\end{samepage}
\normalsize
\par\penalty0\smallskip

\begin{samepage}
\footnotesize
\setlength{\abovedisplayskip}{5pt}
\setlength{\belowdisplayskip}{5pt}
\noindent $\alpha=\frac{3}{2}$.
\[
\begin{aligned}
\mathbf a_{\frac{3}{2}}={}\bigl( & 990\,185\,540\,718\,083,\; -662\,856\,589\,180\,853,\; 274\,592\,288\,924\,910,\; -22\,845\,766\,764\,843,\\
& -48\,514\,218\,242\,985,\; 32\,690\,127\,030\,363,\; -4\,176\,832\,065\,021,\; -4\,411\,876\,558\,221,\\
& 1\,938\,905\,200\,089,\; 964\,094\,691\,889\bigr)
\end{aligned}
\]
\end{samepage}
\normalsize
\par\penalty0\smallskip

\begin{samepage}
\footnotesize
\setlength{\abovedisplayskip}{5pt}
\setlength{\belowdisplayskip}{5pt}
\noindent $\alpha=\frac{13}{8}$.
\[
\begin{aligned}
\mathbf a_{\frac{13}{8}}={}\bigl( & 623\,855\,786\,787\,140,\; -460\,112\,933\,954\,015,\; 221\,862\,471\,254\,028,\; -45\,401\,310\,564\,885,\\
& -24\,932\,129\,857\,668,\; 29\,386\,755\,390\,564,\; -12\,050\,911\,558\,323,\; 1\,095\,283\,460\,700,\\
& 1\,154\,550\,313\,422,\; 120\,300\,699\,727\bigr)
\end{aligned}
\]
\end{samepage}
\normalsize
\par\penalty0\smallskip

\begin{samepage}
\footnotesize
\setlength{\abovedisplayskip}{5pt}
\setlength{\belowdisplayskip}{5pt}
\noindent $\alpha=\frac{7}{4}$.
\[
\begin{aligned}
\mathbf a_{\frac{7}{4}}={}\bigl( & 386\,052\,161\,477\,945,\; -309\,526\,446\,038\,822,\; 167\,110\,396\,679\,061,\; -50\,542\,782\,756\,246,\\
& -7\,872\,032\,368\,035,\; 21\,905\,242\,386\,687,\; -15\,061\,005\,508\,824,\; 5\,976\,347\,895\,441,\\
& -1\,183\,425\,625\,602,\; 94\,308\,578\,499\bigr)
\end{aligned}
\]
\end{samepage}
\normalsize
\par\penalty0\smallskip

\begin{samepage}
\footnotesize
\setlength{\abovedisplayskip}{5pt}
\setlength{\belowdisplayskip}{5pt}
\noindent $\alpha=\frac{15}{8}$.
\[
\begin{aligned}
\mathbf a_{\frac{15}{8}}={}\bigl( & 240\,579\,582\,500\,714,\; -205\,954\,696\,249\,622,\; 121\,904\,943\,756\,453,\; -46\,296\,083\,088\,684,\\
& 2\,248\,972\,170\,291,\; 13\,379\,966\,269\,560,\; -12\,528\,618\,582\,843,\; 6\,810\,095\,792\,313,\\
& -2\,328\,980\,108\,697,\; 413\,592\,746\,619\bigr)
\end{aligned}
\]
\end{samepage}
\normalsize
\par\penalty0\smallskip

\begin{samepage}
\footnotesize
\setlength{\abovedisplayskip}{5pt}
\setlength{\belowdisplayskip}{5pt}
\noindent $\alpha=2$.
\[
\begin{aligned}
\mathbf a_{2}={}\bigl( & 149\,977\,889\,156\,156,\; -135\,130\,411\,136\,447,\; 86\,385\,437\,172\,342,\; -38\,289\,376\,499\,418,\\
& 7\,026\,236\,306\,100,\; 6\,731\,687\,653\,911,\; -8\,605\,760\,197\,914,\; 5\,620\,891\,854\,600,\\
& -2\,337\,628\,263\,795,\; 551\,096\,004\,888\bigr)
\end{aligned}
\]
\end{samepage}
\normalsize
\par\penalty0\smallskip

\begin{samepage}
\footnotesize
\setlength{\abovedisplayskip}{5pt}
\setlength{\belowdisplayskip}{5pt}
\noindent $\alpha=\frac{17}{8}$.
\[
\begin{aligned}
\mathbf a_{\frac{17}{8}}={}\bigl( & 92\,539\,332\,336\,281,\; -86\,900\,722\,827\,623,\; 59\,374\,850\,471\,145,\; -29\,560\,004\,073\,963,\\
& 8\,383\,324\,213\,737,\; 2\,463\,760\,029\,357,\; -5\,252\,197\,617\,711,\; 4\,040\,401\,089\,483,\\
& -1\,907\,480\,962\,827,\; 526\,319\,982\,738\bigr)
\end{aligned}
\]
\end{samepage}
\normalsize
\par\penalty0\smallskip

\begin{samepage}
\footnotesize
\setlength{\abovedisplayskip}{5pt}
\setlength{\belowdisplayskip}{5pt}
\noindent $\alpha=\frac{9}{4}$.
\[
\begin{aligned}
\mathbf a_{\frac{9}{4}}={}\bigl( & 55\,740\,238\,703\,672,\; -54\,147\,946\,412\,825,\; 39\,264\,121\,235\,196,\; -21\,472\,114\,893\,768,\\
& 7\,826\,805\,491\,175,\; 61\,848\,064\,575,\; -2\,813\,494\,442\,418,\; 2\,618\,147\,929\,023,\\
& -1\,380\,000\,266\,172,\; 426\,337\,182\,018\bigr)
\end{aligned}
\]
\end{samepage}
\normalsize
\par\penalty0\smallskip

\subsection{Polynomials and multipliers for the projected bound}\label{subsec:coefficient-projected}
For each index $j$, let $I_j=[\tau_{j,-},\tau_{j,+}]$ be the closed
window specified below, and set
\[
 \sigma=\frac{2\tau-\tau_{j,-}-\tau_{j,+}}{\tau_{j,+}-\tau_{j,-}},
 \qquad \tau\in I_j.
\]
The integer vectors $\mathbf B_{j,0},\mathbf B_{j,1}$ and pairs
$\mathbf G_j=(G_{j,0},G_{j,1})$, $\mathbf E_j=(E_{j,0},E_{j,1})$
define the affine triple by
\[
 \begin{aligned}
 P_\sigma(u)&=10^{-15}\sum_{k=0}^{N_j}
          (B_{j,0,k}+\sigma B_{j,1,k})\ell_k(u),
 &\ell_k(u)&=\frac{L_k^{1/2}(u)}{L_k^{1/2}(0)},\\
 \gamma_\sigma&=10^{-15}(G_{j,0}+\sigma G_{j,1}),
 &\eta_\sigma&=10^{-15}(E_{j,0}+\sigma E_{j,1}).
 \end{aligned}
\]
Within each vector, the index $k$ runs from $0$ to $N_j$.
The Fourier partner $P_\sigma^\sharp$ has Laguerre coefficient
$(-1)^k$ times that of $P_\sigma$.

\begin{lemma}[Projected endpoint polynomials]
\label{lem:projected-endpoints}
Let $\mathscr W$ consist of the thirty affine triples specified below. Their closed rational
windows lie in $[329/100,21]$, their degrees belong to
$\{9,11,15\}$, and the constant and linear coefficients in $\sigma$
in \eqref{eq:affine} belong to $10^{-15}\mathbb Z$. At either endpoint $\sigma=\pm1$
of every window,
\[
 \begin{aligned}
 P_\sigma^\sharp(u)&>0 &&(u\ge0),\\
 P_\sigma(u)+(1+u)^2&<0 &&(u\ge60),\\
 |\gamma_\sigma|&\le4,\qquad |\eta_\sigma|\le1.
 \end{aligned}
\]
\end{lemma}
\begin{proof}
Exact rational Sturm sequences give no roots of $P_\sigma^\sharp$
on $[0,\infty)$ or $P_\sigma+(1+u)^2$ on $[60,\infty)$, with
the asserted signs at $0$ and $60$. Rational comparison verifies
the window inclusions and multiplier bounds using both affine pairs.
\end{proof}

\begin{samepage}
\footnotesize
\setlength{\abovedisplayskip}{5pt}
\setlength{\belowdisplayskip}{5pt}
\noindent $j=0$, $I_j=[\frac{329}{100},\frac{41}{10}]$, $N_j=9$.
\[
\begin{aligned}
\mathbf B_{0,0}={}\bigl( & 17\,569\,160\,072\,547\,198,\; -847\,186\,045\,105\,275,\; -1\,030\,229\,038\,721\,334,\; 241\,223\,247\,774\,188,\\
& 281\,154\,683\,314\,916,\; 41\,220\,793\,898\,547,\; -43\,666\,541\,307\,629,\; -16\,373\,536\,471\,560,\\
& 7\,015\,350\,691\,643,\; 6\,834\,061\,440\,690\bigr)
\end{aligned}
\]
\[
\begin{aligned}
\mathbf B_{0,1}={}\bigl( & 1\,474\,369\,816\,620\,854,\; 660\,458\,930\,735\,836,\; -356\,536\,290\,981\,739,\; 21\,546\,227\,126\,514,\\
& 95\,992\,898\,932\,158,\; 12\,421\,385\,756\,919,\; -14\,639\,689\,931\,025,\; -4\,219\,005\,059\,133,\\
& 2\,327\,630\,563\,695,\; 1\,702\,602\,397\,049\bigr)
\end{aligned}
\]
\[
\mathbf G_{0}=(-356\,592\,375\,773\,652,\quad -62\,929\,038\,482\,252),\qquad
\mathbf E_{0}=(0,\quad 0).
\]
\end{samepage}
\normalsize
\par\penalty0\smallskip

\begin{samepage}
\footnotesize
\setlength{\abovedisplayskip}{5pt}
\setlength{\belowdisplayskip}{5pt}
\noindent $j=1$, $I_j=[4,5]$, $N_j=9$.
\[
\begin{aligned}
\mathbf B_{1,0}={}\bigl( & 19\,316\,776\,696\,995\,564,\; 242\,523\,784\,457\,620,\; -1\,653\,047\,234\,696\,646,\; 305\,718\,572\,157\,493,\\
& 348\,109\,144\,647\,535,\; 7\,861\,989\,254\,154,\; -57\,858\,206\,188\,501,\; -7\,381\,879\,167\,890,\\
& 14\,328\,060\,410\,484,\; 8\,436\,758\,143\,502\bigr)
\end{aligned}
\]
\[
\begin{aligned}
\mathbf B_{1,1}={}\bigl( & 807\,200\,581\,436\,003,\; 608\,841\,649\,515\,109,\; -407\,089\,280\,723\,188,\; 70\,966\,283\,311\,898,\\
& 5\,965\,704\,940\,330,\; -49\,285\,831\,258\,677,\; -6\,230\,162\,130\,759,\; 13\,335\,377\,168\,505,\\
& 6\,482\,749\,856\,404,\; 266\,156\,547\,251\bigr)
\end{aligned}
\]
\[
\mathbf G_{1}=(-471\,990\,348\,950\,231,\quad -70\,494\,054\,429\,959),\qquad
\mathbf E_{1}=(0,\quad 0).
\]
\end{samepage}
\normalsize
\par\penalty0\smallskip

\begin{samepage}
\footnotesize
\setlength{\abovedisplayskip}{5pt}
\setlength{\belowdisplayskip}{5pt}
\noindent $j=2$, $I_j=[\frac{49}{10},\frac{11}{2}]$, $N_j=9$.
\[
\begin{aligned}
\mathbf B_{2,0}={}\bigl( & 21\,507\,715\,265\,165\,352,\; 1\,283\,114\,462\,926\,282,\; -1\,925\,069\,513\,653\,568,\; 187\,951\,008\,959\,265,\\
& 549\,995\,073\,115\,716,\; 115\,935\,380\,898\,583,\; -93\,343\,914\,384\,373,\; -36\,262\,763\,435\,869,\\
& 17\,294\,149\,076\,150,\; 14\,425\,608\,345\,904\bigr)
\end{aligned}
\]
\[
\begin{aligned}
\mathbf B_{2,1}={}\bigl( & 2\,314\,962\,850\,359\,984,\; 737\,603\,274\,488\,664,\; -379\,401\,785\,140\,896,\; -13\,785\,750\,521\,230,\\
& 100\,875\,502\,097\,992,\; 20\,184\,631\,098\,661,\; -27\,483\,250\,093\,765,\; -5\,409\,899\,257\,746,\\
& 7\,949\,304\,435\,895,\; 3\,048\,427\,152\,146\bigr)
\end{aligned}
\]
\[
\mathbf G_{2}=(-575\,006\,717\,891\,784,\quad -58\,150\,015\,978\,538),\qquad
\mathbf E_{2}=(0,\quad 0).
\]
\end{samepage}
\normalsize
\par\penalty0\smallskip

\begin{samepage}
\footnotesize
\setlength{\abovedisplayskip}{5pt}
\setlength{\belowdisplayskip}{5pt}
\noindent $j=3$, $I_j=[\frac{27}{5},\frac{61}{10}]$, $N_j=9$.
\[
\begin{aligned}
\mathbf B_{3,0}={}\bigl( & 19\,285\,552\,098\,818\,868,\; 1\,094\,234\,783\,762\,215,\; -1\,688\,139\,636\,355\,250,\; 193\,603\,249\,695\,094,\\
& 609\,591\,975\,092\,513,\; 198\,899\,567\,899\,884,\; -85\,847\,444\,193\,574,\; -39\,209\,550\,022\,969,\\
& -24\,940\,059\,437\,963,\; 60\,809\,360\,013\,190\bigr)
\end{aligned}
\]
\[
\begin{aligned}
\mathbf B_{3,1}={}\bigl( & 325\,295\,964\,960\,971,\; 249\,205\,790\,679\,789,\; -57\,224\,115\,862\,622,\; 5\,894\,514\,248\,189,\\
& 128\,202\,095\,427\,927,\; 91\,490\,510\,319\,199,\; -15\,129\,411\,421\,212,\; -5\,218\,369\,889\,627,\\
& -34\,973\,697\,377\,690,\; 48\,336\,517\,773\,658\bigr)
\end{aligned}
\]
\[
\mathbf G_{3}=(-616\,213\,073\,263\,267,\quad -43\,184\,204\,551\,397),\qquad
\mathbf E_{3}=(0,\quad 0).
\]
\end{samepage}
\normalsize
\par\penalty0\smallskip

\begin{samepage}
\footnotesize
\setlength{\abovedisplayskip}{5pt}
\setlength{\belowdisplayskip}{5pt}
\noindent $j=4$, $I_j=[\frac{59}{10},7]$, $N_j=9$.
\[
\begin{aligned}
\mathbf B_{4,0}={}\bigl( & 15\,722\,571\,687\,964\,236,\; 1\,039\,797\,073\,593\,928,\; -1\,344\,356\,031\,009\,688,\; 96\,452\,138\,362\,805,\\
& 338\,394\,306\,155\,240,\; 76\,631\,884\,197\,823,\; -42\,257\,337\,950\,306,\; -20\,287\,777\,948\,815,\\
& 5\,447\,034\,721\,715,\; 6\,750\,923\,584\,069\bigr)
\end{aligned}
\]
\[
\begin{aligned}
\mathbf B_{4,1}={}\bigl( & -2\,863\,129\,464\,196\,410,\; -62\,630\,961\,569\,201,\; 327\,795\,394\,716\,658,\; -78\,835\,717\,152\,385,\\
& -161\,615\,007\,550\,752,\; -38\,391\,247\,191\,358,\; 31\,317\,044\,059\,642,\; 15\,686\,085\,184\,577,\\
& -5\,002\,095\,260\,011,\; -6\,402\,844\,773\,094\bigr)
\end{aligned}
\]
\[
\mathbf G_{4}=(-644\,058\,934\,273\,268,\quad -20\,407\,970\,301\,288),\qquad
\mathbf E_{4}=(0,\quad 0).
\]
\end{samepage}
\normalsize
\par\penalty0\smallskip

\begin{samepage}
\footnotesize
\setlength{\abovedisplayskip}{5pt}
\setlength{\belowdisplayskip}{5pt}
\noindent $j=5$, $I_j=[\frac{69}{10},\frac{17}{2}]$, $N_j=9$.
\[
\begin{aligned}
\mathbf B_{5,0}={}\bigl( & 13\,541\,837\,527\,757\,038,\; 1\,015\,487\,174\,994\,281,\; -1\,290\,765\,737\,796\,498,\; 130\,113\,806\,791\,715,\\
& 398\,699\,400\,050\,781,\; 87\,211\,982\,185\,334,\; -69\,022\,215\,723\,988,\; -21\,759\,591\,815\,508,\\
& 18\,068\,197\,302\,621,\; 31\,298\,501\,639\,941\bigr)
\end{aligned}
\]
\[
\begin{aligned}
\mathbf B_{5,1}={}\bigl( & -3\,807\,880\,988\,574\,810,\; -367\,553\,157\,619\,942,\; 375\,759\,596\,611\,259,\; -22\,699\,433\,641\,844,\\
& -108\,848\,625\,405\,022,\; -34\,131\,419\,166\,791,\; 5\,782\,388\,501\,435,\; 15\,651\,200\,344\,973,\\
& 7\,443\,802\,576\,253,\; 17\,673\,413\,907\,249\bigr)
\end{aligned}
\]
\[
\mathbf G_{5}=(-681\,226\,328\,212\,904,\quad 27\,485\,075\,077\,900),\qquad
\mathbf E_{5}=(0,\quad 0).
\]
\end{samepage}
\normalsize
\par\penalty0\smallskip

\begin{samepage}
\footnotesize
\setlength{\abovedisplayskip}{5pt}
\setlength{\belowdisplayskip}{5pt}
\noindent $j=6$, $I_j=[\frac{52}{5},12]$, $N_j=9$.
\[
\begin{aligned}
\mathbf B_{6,0}={}\bigl( & 19\,962\,117\,251\,205\,448,\; 3\,545\,436\,732\,682\,850,\; -1\,969\,448\,228\,725\,096,\; -163\,885\,113\,838\,356,\\
& 600\,813\,371\,833\,352,\; 277\,244\,219\,080\,790,\; -35\,415\,949\,816\,078,\; -45\,869\,970\,661\,484,\\
& -1\,887\,017\,046\,914,\; 11\,000\,818\,886\,699\bigr)
\end{aligned}
\]
\[
\begin{aligned}
\mathbf B_{6,1}={}\bigl( & -6\,939\,123\,160\,616\,658,\; -1\,584\,377\,757\,283\,264,\; 622\,472\,539\,673\,208,\; 151\,265\,387\,512\,468,\\
& -183\,295\,337\,832\,780,\; -99\,003\,197\,881\,982,\; 23\,470\,167\,415\,634,\; 38\,486\,588\,977\,045,\\
& 2\,069\,745\,185\,361,\; -10\,843\,932\,908\,385\bigr)
\end{aligned}
\]
\[
\mathbf G_{6}=(-972\,692\,452\,240\,126,\quad 35\,138\,309\,985\,537),\qquad
\mathbf E_{6}=(0,\quad 0).
\]
\end{samepage}
\normalsize
\par\penalty0\smallskip

\begin{samepage}
\footnotesize
\setlength{\abovedisplayskip}{5pt}
\setlength{\belowdisplayskip}{5pt}
\noindent $j=7$, $I_j=[\frac{119}{10},\frac{68}{5}]$, $N_j=9$.
\[
\begin{aligned}
\mathbf B_{7,0}={}\bigl( & 26\,424\,248\,927\,604\,356,\; 5\,652\,590\,986\,039\,579,\; -2\,659\,457\,449\,082\,472,\; -332\,014\,324\,277\,028,\\
& 864\,694\,277\,110\,567,\; 410\,962\,111\,936\,891,\; -101\,817\,259\,172\,916,\; -91\,533\,856\,795\,605,\\
& -17\,276\,928\,030\,913,\; 90\,604\,471\,780\,598\bigr)
\end{aligned}
\]
\[
\begin{aligned}
\mathbf B_{7,1}={}\bigl( & 1\,118\,788\,943\,300\,847,\; 524\,408\,394\,408\,487,\; -189\,785\,302\,729\,921,\; -21\,073\,243\,643\,816,\\
& 105\,984\,211\,906\,609,\; 44\,363\,711\,829\,112,\; -44\,772\,721\,825\,239,\; -9\,333\,523\,887\,696,\\
& -13\,426\,561\,900\,353,\; 69\,286\,556\,148\,052\bigr)
\end{aligned}
\]
\[
\mathbf G_{7}=(-1\,116\,782\,167\,221\,306,\quad -61\,667\,307\,131\,519),\qquad
\mathbf E_{7}=(0,\quad 0).
\]
\end{samepage}
\normalsize
\par\penalty0\smallskip

\begin{samepage}
\footnotesize
\setlength{\abovedisplayskip}{5pt}
\setlength{\belowdisplayskip}{5pt}
\noindent $j=8$, $I_j=[\frac{94}{25},\frac{21}{5}]$, $N_j=11$.
\[
\begin{aligned}
\mathbf B_{8,0}={}\bigl( & 22\,508\,290\,242\,120\,828,\; 646\,988\,994\,763\,625,\; -1\,371\,685\,408\,032\,282,\; 5\,714\,793\,022\,501,\\
& 373\,863\,294\,392\,758,\; 168\,840\,203\,337\,073,\; -4\,792\,399\,292\,133,\; -34\,824\,763\,682\,570,\\
& -9\,853\,332\,769\,729,\; 7\,427\,711\,050\,701,\; 7\,256\,199\,808\,504,\; 2\,340\,499\,109\,608\bigr)
\end{aligned}
\]
\[
\begin{aligned}
\mathbf B_{8,1}={}\bigl( & 1\,022\,834\,684\,652\,179,\; 452\,895\,494\,851\,332,\; -132\,068\,480\,055\,398,\; -35\,768\,421\,403\,981,\\
& 60\,576\,694\,785\,904,\; 44\,523\,622\,144\,433,\; 4\,946\,265\,955\,872,\; -8\,327\,937\,776\,012,\\
& -3\,761\,693\,877\,268,\; 1\,735\,880\,366\,610,\; 2\,252\,328\,157\,903,\; 798\,901\,263\,388\bigr)
\end{aligned}
\]
\[
\mathbf G_{8}=(-416\,340\,904\,217\,397,\quad -35\,071\,862\,128\,564),\qquad
\mathbf E_{8}=(0,\quad 0).
\]
\end{samepage}
\normalsize
\par\penalty0\smallskip

\begin{samepage}
\footnotesize
\setlength{\abovedisplayskip}{5pt}
\setlength{\belowdisplayskip}{5pt}
\noindent $j=9$, $I_j=[\frac{41}{10},\frac{47}{10}]$, $N_j=11$.
\[
\begin{aligned}
\mathbf B_{9,0}={}\bigl( & 23\,990\,410\,355\,527\,180,\; 1\,389\,218\,208\,732\,794,\; -1\,580\,672\,908\,969\,926,\; -51\,618\,168\,218\,653,\\
& 468\,668\,897\,863\,708,\; 237\,138\,649\,332\,626,\; 2\,217\,788\,878\,505,\; -47\,632\,303\,362\,365,\\
& -15\,462\,219\,473\,605,\; 10\,091\,228\,140\,648,\; 10\,643\,410\,215\,524,\; 3\,537\,193\,352\,231\bigr)
\end{aligned}
\]
\[
\begin{aligned}
\mathbf B_{9,1}={}\bigl( & 937\,468\,468\,400\,903,\; 496\,750\,702\,632\,597,\; -138\,369\,598\,311\,307,\; -38\,688\,488\,611\,182,\\
& 62\,806\,221\,922\,965,\; 45\,295\,655\,371\,739,\; 4\,617\,364\,673\,034,\; -8\,494\,220\,513\,653,\\
& -3\,711\,207\,819\,962,\; 1\,768\,383\,835\,771,\; 2\,238\,664\,095\,470,\; 785\,572\,773\,613\bigr)
\end{aligned}
\]
\[
\mathbf G_{9}=(-478\,496\,716\,355\,679,\quad -43\,179\,914\,146\,186),\qquad
\mathbf E_{9}=(0,\quad 0).
\]
\end{samepage}
\normalsize
\par\penalty0\smallskip

\begin{samepage}
\footnotesize
\setlength{\abovedisplayskip}{5pt}
\setlength{\belowdisplayskip}{5pt}
\noindent $j=10$, $I_j=[\frac{23}{5},\frac{101}{20}]$, $N_j=11$.
\[
\begin{aligned}
\mathbf B_{10,0}={}\bigl( & 25\,391\,315\,756\,982\,704,\; 2\,091\,846\,874\,613\,638,\; -1\,841\,898\,047\,383\,381,\; -92\,205\,019\,528\,323,\\
& 519\,386\,034\,503\,886,\; 244\,430\,196\,705\,725,\; -15\,422\,536\,396\,783,\; -46\,108\,905\,006\,177,\\
& -8\,512\,851\,025\,774,\; 8\,417\,417\,286\,718,\; 6\,103\,465\,950\,138,\; 1\,804\,982\,474\,404\bigr)
\end{aligned}
\]
\[
\begin{aligned}
\mathbf B_{10,1}={}\bigl( & 663\,836\,874\,887\,466,\; 358\,287\,568\,541\,485,\; -139\,152\,030\,681\,834,\; -20\,581\,870\,504\,616,\\
& 9\,531\,455\,209\,083,\; -9\,220\,273\,862\,189,\; -12\,908\,207\,598\,116,\; 6\,259\,222\,925\,854,\\
& 6\,731\,768\,637\,941,\; -3\,218\,046\,555\,562,\; -5\,185\,892\,448\,235,\; -1\,752\,063\,730\,913\bigr)
\end{aligned}
\]
\[
\mathbf G_{10}=(-542\,121\,689\,317\,240,\quad -32\,132\,930\,263\,982),\qquad
\mathbf E_{10}=(0,\quad 0).
\]
\end{samepage}
\normalsize
\par\penalty0\smallskip

\begin{samepage}
\footnotesize
\setlength{\abovedisplayskip}{5pt}
\setlength{\belowdisplayskip}{5pt}
\noindent $j=11$, $I_j=[5,\frac{11}{2}]$, $N_j=11$.
\[
\begin{aligned}
\mathbf B_{11,0}={}\bigl( & 25\,697\,939\,325\,139\,504,\; 2\,460\,868\,167\,415\,433,\; -1\,907\,861\,206\,211\,602,\; -129\,824\,359\,314\,492,\\
& 593\,174\,238\,640\,382,\; 296\,326\,429\,522\,767,\; -16\,061\,467\,543\,140,\; -68\,755\,031\,554\,619,\\
& -11\,875\,944\,956\,253,\; 25\,816\,978\,060\,205,\; 22\,279\,619\,106\,589,\; 8\,106\,635\,521\,248\bigr)
\end{aligned}
\]
\[
\begin{aligned}
\mathbf B_{11,1}={}\bigl( & -233\,267\,204\,259\,728,\; 119\,162\,985\,189\,446,\; -10\,787\,601\,517\,903,\; -13\,781\,082\,029\,282,\\
& 6\,071\,538\,568\,768,\; -2\,275\,261\,666\,861,\; -11\,903\,905\,118\,580,\; -6\,066\,660\,100\,714,\\
& 5\,484\,134\,587\,194,\; 11\,308\,038\,724\,095,\; 8\,801\,105\,774\,883,\; 4\,005\,446\,842\,765\bigr)
\end{aligned}
\]
\[
\mathbf G_{11}=(-588\,624\,008\,734\,749,\quad -19\,514\,570\,439\,558),\qquad
\mathbf E_{11}=(0,\quad 0).
\]
\end{samepage}
\normalsize
\par\penalty0\smallskip

\begin{samepage}
\footnotesize
\setlength{\abovedisplayskip}{5pt}
\setlength{\belowdisplayskip}{5pt}
\noindent $j=12$, $I_j=[\frac{27}{5},\frac{13}{2}]$, $N_j=11$.
\[
\begin{aligned}
\mathbf B_{12,0}={}\bigl( & 25\,267\,995\,938\,725\,460,\; 2\,791\,102\,243\,607\,135,\; -2\,052\,316\,426\,661\,384,\; -106\,920\,723\,140\,700,\\
& 603\,169\,576\,171\,931,\; 278\,688\,135\,527\,074,\; -11\,411\,230\,269\,687,\; -55\,976\,183\,357\,951,\\
& -13\,686\,716\,449\,575,\; 11\,399\,881\,169\,827,\; 9\,796\,131\,649\,101,\; 2\,797\,772\,189\,620\bigr)
\end{aligned}
\]
\[
\begin{aligned}
\mathbf B_{12,1}={}\bigl( & -1\,571\,320\,597\,117\,049,\; -78\,060\,730\,314\,835,\; -13\,168\,821\,015\,957,\; 48\,632\,487\,327\,079,\\
& -45\,441\,253\,477\,856,\; -59\,884\,785\,150\,432,\; -8\,497\,642\,510\,811,\; 14\,985\,352\,782\,709,\\
& 6\,187\,698\,033\,991,\; -5\,106\,547\,103\,247,\; -5\,566\,084\,786\,089,\; -1\,865\,529\,929\,909\bigr)
\end{aligned}
\]
\[
\mathbf G_{12}=(-659\,734\,670\,683\,471,\quad -37\,587\,675\,437\,251),\qquad
\mathbf E_{12}=(0,\quad 0).
\]
\end{samepage}
\normalsize
\par\penalty0\smallskip

\begin{samepage}
\footnotesize
\setlength{\abovedisplayskip}{5pt}
\setlength{\belowdisplayskip}{5pt}
\noindent $j=13$, $I_j=[\frac{32}{5},\frac{15}{2}]$, $N_j=11$.
\[
\begin{aligned}
\mathbf B_{13,0}={}\bigl( & 30\,086\,308\,538\,164\,296,\; 4\,697\,884\,607\,504\,267,\; -2\,517\,664\,298\,360\,390,\; -309\,544\,929\,231\,317,\\
& 884\,795\,141\,548\,105,\; 491\,695\,327\,104\,811,\; -15\,767\,851\,904\,839,\; -110\,409\,338\,623\,988,\\
& -12\,675\,337\,162\,481,\; 65\,498\,921\,071\,088,\; 60\,029\,743\,021\,923,\; 29\,417\,919\,432\,189\bigr)
\end{aligned}
\]
\[
\begin{aligned}
\mathbf B_{13,1}={}\bigl( & 1\,896\,574\,050\,129\,116,\; 849\,832\,147\,297\,293,\; -241\,552\,798\,540\,493,\; -79\,723\,881\,986\,277,\\
& 123\,065\,907\,917\,517,\; 76\,457\,830\,377\,002,\; -16\,524\,638\,587\,453,\; -23\,883\,747\,590\,441,\\
& 12\,321\,031\,248\,981,\; 45\,324\,482\,569\,540,\; 41\,034\,327\,450\,599,\; 23\,643\,196\,304\,737\bigr)
\end{aligned}
\]
\[
\mathbf G_{13}=(-823\,210\,101\,690\,386,\quad -85\,255\,127\,499\,723),\qquad
\mathbf E_{13}=(0,\quad 0).
\]
\end{samepage}
\normalsize
\par\penalty0\smallskip

\begin{samepage}
\footnotesize
\setlength{\abovedisplayskip}{5pt}
\setlength{\belowdisplayskip}{5pt}
\noindent $j=14$, $I_j=[\frac{37}{5},9]$, $N_j=11$.
\[
\begin{aligned}
\mathbf B_{14,0}={}\bigl( & 31\,170\,360\,609\,306\,396,\; 5\,532\,913\,610\,914\,782,\; -2\,625\,292\,786\,204\,126,\; -388\,519\,148\,313\,454,\\
& 1\,040\,387\,882\,886\,959,\; 667\,814\,173\,834\,645,\; 45\,050\,285\,655\,972,\; -132\,868\,877\,361\,956,\\
& -47\,498\,143\,113\,654,\; 66\,813\,215\,037\,796,\; 61\,529\,048\,673\,254,\; 54\,071\,173\,485\,244\bigr)
\end{aligned}
\]
\[
\begin{aligned}
\mathbf B_{14,1}={}\bigl( & 2\,500\,084\,699\,109\,088,\; 1\,051\,403\,203\,396\,738,\; -205\,011\,947\,764\,261,\; -106\,684\,356\,520\,441,\\
& 208\,952\,645\,159\,678,\; 204\,005\,369\,476\,402,\; 41\,918\,403\,758\,049,\; -36\,398\,213\,875\,217,\\
& -19\,281\,795\,237\,158,\; 44\,239\,713\,416\,139,\; 40\,193\,969\,295\,018,\; 47\,585\,881\,613\,859\bigr)
\end{aligned}
\]
\[
\mathbf G_{14}=(-938\,195\,656\,364\,055,\quad -106\,283\,021\,446\,540),\qquad
\mathbf E_{14}=(0,\quad 0).
\]
\end{samepage}
\normalsize
\par\penalty0\smallskip

\begin{samepage}
\footnotesize
\setlength{\abovedisplayskip}{5pt}
\setlength{\belowdisplayskip}{5pt}
\noindent $j=15$, $I_j=[\frac{89}{10},\frac{21}{2}]$, $N_j=9$.
\[
\begin{aligned}
\mathbf B_{15,0}={}\bigl( & 29\,860\,030\,802\,049\,720,\; 5\,772\,178\,250\,774\,631,\; -2\,897\,533\,367\,861\,917,\; -326\,030\,444\,049\,826,\\
& 931\,514\,412\,172\,014,\; 434\,888\,357\,655\,536,\; -104\,110\,789\,973\,341,\; -99\,537\,430\,724\,009,\\
& -20\,633\,644\,101\,026,\; 85\,921\,754\,564\,972\bigr)
\end{aligned}
\]
\[
\begin{aligned}
\mathbf B_{15,1}={}\bigl( & 1\,736\,079\,591\,229\,595,\; 774\,470\,070\,917\,232,\; -292\,991\,926\,351\,255,\; -32\,180\,697\,511\,298,\\
& 147\,178\,780\,747\,502,\; 62\,578\,730\,103\,754,\; -45\,263\,197\,101\,846,\; -16\,059\,975\,906\,373,\\
& -16\,718\,627\,613\,336,\; 64\,344\,261\,301\,752\bigr)
\end{aligned}
\]
\[
\mathbf G_{15}=(-1\,025\,270\,863\,284\,464,\quad -89\,026\,216\,023\,101),\qquad
\mathbf E_{15}=(0,\quad 0).
\]
\end{samepage}
\normalsize
\par\penalty0\smallskip

\begin{samepage}
\footnotesize
\setlength{\abovedisplayskip}{5pt}
\setlength{\belowdisplayskip}{5pt}
\noindent $j=16$, $I_j=[\frac{54}{5},13]$, $N_j=9$.
\[
\begin{aligned}
\mathbf B_{16,0}={}\bigl( & 120\,423\,482\,091\,707\,192,\; 26\,184\,554\,857\,755\,136,\; -10\,689\,574\,320\,380\,482,\; -2\,169\,982\,205\,706\,272,\\
& 3\,107\,946\,444\,186\,794,\; 1\,921\,702\,348\,779\,464,\; 84\,086\,166\,290\,651,\; -219\,943\,311\,556\,045,\\
& -33\,211\,948\,575\,614,\; 47\,335\,727\,004\,217\bigr)
\end{aligned}
\]
\[
\begin{aligned}
\mathbf B_{16,1}={}\bigl( & -9\,859\,710\,295\,033\,688,\; -1\,748\,861\,379\,149\,857,\; 776\,907\,128\,927\,267,\; 206\,340\,195\,033\,300,\\
& -246\,697\,890\,347\,827,\; -72\,480\,732\,746\,922,\; 190\,526\,956\,111\,355,\; 175\,007\,026\,984\,950,\\
& 23\,064\,463\,250\,545,\; -45\,731\,839\,308\,256\bigr)
\end{aligned}
\]
\[
\mathbf G_{16}=(-3\,992\,229\,794\,907\,459,\quad 7\,770\,205\,092\,541),\qquad
\mathbf E_{16}=(-243\,073\,071\,072\,496,\quad 28\,398\,997\,812\,925).
\]
\end{samepage}
\normalsize
\par\penalty0\smallskip

\begin{samepage}
\footnotesize
\setlength{\abovedisplayskip}{5pt}
\setlength{\belowdisplayskip}{5pt}
\noindent $j=17$, $I_j=[\frac{64}{5},15]$, $N_j=9$.
\[
\begin{aligned}
\mathbf B_{17,0}={}\bigl( & 104\,068\,319\,993\,649\,240,\; 24\,008\,939\,108\,635\,688,\; -8\,904\,209\,837\,630\,534,\; -2\,181\,329\,325\,308\,189,\\
& 2\,395\,933\,585\,955\,964,\; 1\,412\,411\,876\,511\,648,\; -170\,987\,752\,273\,108,\; -306\,390\,479\,976\,519,\\
& -20\,012\,541\,617\,307,\; 149\,659\,854\,688\,174\bigr)
\end{aligned}
\]
\[
\begin{aligned}
\mathbf B_{17,1}={}\bigl( & -12\,108\,514\,444\,964\,706,\; -2\,475\,780\,276\,357\,371,\; 1\,511\,520\,909\,340\,921,\; 78\,766\,991\,424\,931,\\
& -697\,676\,602\,143\,848,\; -456\,920\,263\,851\,166,\; -82\,384\,647\,718\,634,\; 61\,179\,638\,073\,456,\\
& 34\,563\,007\,047\,673,\; 63\,316\,640\,036\,743\bigr)
\end{aligned}
\]
\[
\mathbf G_{17}=(-4\,000\,000\,000\,000\,000,\quad 0),\qquad
\mathbf E_{17}=(-199\,077\,969\,971\,347,\quad 7\,138\,091\,358\,430).
\]
\end{samepage}
\normalsize
\par\penalty0\smallskip

\begin{samepage}
\footnotesize
\setlength{\abovedisplayskip}{5pt}
\setlength{\belowdisplayskip}{5pt}
\noindent $j=18$, $I_j=[\frac{74}{5},17]$, $N_j=9$.
\[
\begin{aligned}
\mathbf B_{18,0}={}\bigl( & 102\,159\,525\,973\,904\,224,\; 25\,314\,265\,149\,168\,144,\; -9\,851\,807\,381\,129\,676,\; -2\,069\,803\,055\,521\,264,\\
& 2\,655\,477\,784\,451\,309,\; 1\,376\,820\,942\,511\,354,\; -321\,037\,694\,689\,852,\; -289\,269\,091\,722\,315,\\
& 104\,140\,796\,559\,105,\; 199\,248\,916\,292\,729\bigr)
\end{aligned}
\]
\[
\begin{aligned}
\mathbf B_{18,1}={}\bigl( & -1\,354\,412\,418\,436\,295,\; 777\,868\,603\,510\,280,\; -508\,404\,681\,706\,166,\; 18\,688\,174\,637\,421,\\
& -155\,090\,122\,511\,417,\; -342\,437\,089\,060\,120,\; -248\,713\,604\,938\,061,\; 46\,346\,417\,031\,264,\\
& 155\,875\,992\,494\,585,\; 120\,707\,249\,144\,485\bigr)
\end{aligned}
\]
\[
\mathbf G_{18}=(-3\,999\,999\,999\,999\,976,\quad 0),\qquad
\mathbf E_{18}=(-147\,440\,377\,028\,222,\quad 22\,692\,972\,134\,757).
\]
\end{samepage}
\normalsize
\par\penalty0\smallskip

\begin{samepage}
\footnotesize
\setlength{\abovedisplayskip}{5pt}
\setlength{\belowdisplayskip}{5pt}
\noindent $j=19$, $I_j=[\frac{84}{5},\frac{39}{2}]$, $N_j=9$.
\[
\begin{aligned}
\mathbf B_{19,0}={}\bigl( & 92\,871\,740\,113\,613\,624,\; 23\,795\,746\,939\,702\,472,\; -8\,378\,381\,255\,374\,338,\; -2\,097\,676\,720\,777\,805,\\
& 2\,506\,933\,812\,429\,303,\; 1\,581\,023\,685\,407\,946,\; -193\,470\,940\,263\,773,\; -371\,938\,459\,793\,276,\\
& -28\,574\,349\,255\,691,\; 189\,902\,557\,774\,830\bigr)
\end{aligned}
\]
\[
\begin{aligned}
\mathbf B_{19,1}={}\bigl( & 1\,430\,382\,152\,979\,440,\; 1\,558\,095\,390\,431\,998,\; -107\,931\,152\,792\,350,\; -208\,865\,449\,207\,216,\\
& -5\,684\,368\,266\,164,\; 28\,270\,773\,119\,052,\; -136\,127\,709\,391\,537,\; -71\,125\,039\,006\,358,\\
& 19\,523\,005\,561\,192,\; 119\,794\,184\,133\,576\bigr)
\end{aligned}
\]
\[
\mathbf G_{19}=(-4\,000\,000\,000\,000\,000,\quad 0),\qquad
\mathbf E_{19}=(-127\,616\,962\,123\,000,\quad 20\,426\,985\,468\,678).
\]
\end{samepage}
\normalsize
\par\penalty0\smallskip

\begin{samepage}
\footnotesize
\setlength{\abovedisplayskip}{5pt}
\setlength{\belowdisplayskip}{5pt}
\noindent $j=20$, $I_j=[\frac{201}{50},\frac{17}{4}]$, $N_j=15$.
\[
\begin{aligned}
\mathbf B_{20,0}={}\bigl( & 4\,860\,254\,620\,684\,765,\; 3\,509\,474\,372\,118\,688,\; -1\,031\,383\,355\,255\,547,\; -216\,597\,120\,470\,127,\\
& 300\,306\,038\,736\,491,\; 171\,953\,048\,005\,878,\; -912\,916\,865\,971,\; -36\,436\,547\,930\,836,\\
& -11\,220\,921\,448\,266,\; 7\,829\,310\,506\,726,\; 6\,531\,161\,822\,212,\; -202\,838\,064\,504,\\
& -2\,435\,187\,057\,328,\; -298\,126\,992\,464,\; 1\,116\,124\,847\,284,\; 693\,270\,891\,642\bigr)
\end{aligned}
\]
\[
\begin{aligned}
\mathbf B_{20,1}={}\bigl( & -713\,405\,384\,469\,004,\; 44\,460\,311\,224\,227,\; 58\,966\,176\,024\,799,\; -19\,359\,454\,430\,388,\\
& 243\,045\,245\,915,\; 8\,160\,014\,324\,494,\; 2\,141\,596\,079\,627,\; -1\,364\,916\,262\,690,\\
& -925\,721\,288\,199,\; 215\,536\,805\,831,\; 382\,915\,335\,164,\; 51\,733\,404\,608,\\
& -128\,602\,645\,487,\; -35\,845\,648\,798,\; 52\,842\,666\,002,\; 39\,674\,930\,174\bigr)
\end{aligned}
\]
\[
\mathbf G_{20}=(-279\,255\,413\,970\,420,\quad 9\,273\,754\,016\,918),\qquad
\mathbf E_{20}=(366\,425\,297\,587\,241,\quad -203\,062\,634\,486).
\]
\end{samepage}
\normalsize
\par\penalty0\smallskip

\begin{samepage}
\footnotesize
\setlength{\abovedisplayskip}{5pt}
\setlength{\belowdisplayskip}{5pt}
\noindent $j=21$, $I_j=[\frac{81}{20},\frac{17}{4}]$, $N_j=15$.
\[
\begin{aligned}
\mathbf B_{21,0}={}\bigl( & 4\,724\,681\,345\,135\,715,\; 3\,517\,446\,096\,460\,463,\; -1\,022\,204\,406\,495\,594,\; -219\,308\,306\,235\,555,\\
& 300\,047\,162\,503\,649,\; 172\,762\,877\,010\,806,\; -734\,544\,843\,324,\; -36\,594\,898\,958\,980,\\
& -11\,299\,910\,081\,030,\; 7\,870\,604\,686\,635,\; 6\,574\,004\,182\,692,\; -202\,745\,408\,770,\\
& -2\,451\,594\,599\,244,\; -300\,303\,773\,525,\; 1\,123\,878\,428\,234,\; 698\,099\,233\,759\bigr)
\end{aligned}
\]
\[
\begin{aligned}
\mathbf B_{21,1}={}\bigl( & -569\,075\,860\,159\,396,\; 44\,160\,924\,092\,537,\; 47\,796\,650\,731\,920,\; -17\,239\,896\,332\,091,\\
& 1\,175\,593\,230\,625,\; 7\,737\,609\,755\,220,\; 1\,922\,872\,969\,334,\; -1\,309\,316\,427\,839,\\
& -863\,901\,557\,563,\; 207\,764\,852\,113,\; 360\,652\,704\,649,\; 48\,435\,222\,094,\\
& -121\,215\,755\,032,\; -34\,075\,628\,488,\; 49\,554\,618\,730,\; 37\,402\,465\,986\bigr)
\end{aligned}
\]
\[
\mathbf G_{21}=(-277\,577\,527\,549\,590,\quad 7\,151\,707\,212\,092),\qquad
\mathbf E_{21}=(367\,082\,577\,002\,698,\quad -664\,660\,607\,657).
\]
\end{samepage}
\normalsize
\par\penalty0\smallskip

\begin{samepage}
\footnotesize
\setlength{\abovedisplayskip}{5pt}
\setlength{\belowdisplayskip}{5pt}
\noindent $j=22$, $I_j=[\frac{21}{5},\frac{89}{20}]$, $N_j=15$.
\[
\begin{aligned}
\mathbf B_{22,0}={}\bigl( & 4\,895\,063\,265\,871\,995,\; 3\,434\,007\,436\,155\,180,\; -946\,435\,127\,444\,308,\; -242\,030\,363\,261\,259,\\
& 296\,954\,237\,757\,154,\; 186\,168\,417\,107\,524,\; 5\,983\,961\,779\,026,\; -37\,495\,508\,280\,519,\\
& -13\,573\,664\,562\,111,\; 7\,403\,059\,038\,300,\; 7\,190\,867\,345\,604,\; 206\,513\,487\,439,\\
& -2\,581\,331\,909\,599,\; -492\,606\,923\,716,\; 1\,106\,838\,569\,482,\; 747\,916\,690\,900\bigr)
\end{aligned}
\]
\[
\begin{aligned}
\mathbf B_{22,1}={}\bigl( & -936\,049\,063\,393\,733,\; 51\,128\,421\,286\,063,\; 82\,811\,393\,442\,207,\; -22\,762\,456\,345\,969,\\
& -9\,575\,838\,133\,323,\; 7\,311\,857\,948\,235,\; 5\,190\,816\,751\,348,\; -64\,653\,572\,238,\\
& -1\,483\,956\,078\,818,\; -457\,490\,512\,923,\; 344\,728\,809\,864,\; 300\,904\,363\,793,\\
& -53\,742\,029\,167,\; -135\,690\,333\,021,\; -30\,638\,662\,184,\; 26\,981\,265\,647\bigr)
\end{aligned}
\]
\[
\mathbf G_{22}=(-271\,292\,404\,116\,455,\quad 14\,598\,055\,172\,568),\qquad
\mathbf E_{22}=(343\,691\,100\,407\,056,\quad 2\,727\,036\,787\,781).
\]
\end{samepage}
\normalsize
\par\penalty0\smallskip

\begin{samepage}
\footnotesize
\setlength{\abovedisplayskip}{5pt}
\setlength{\belowdisplayskip}{5pt}
\noindent $j=23$, $I_j=[\frac{21}{5},\frac{9}{2}]$, $N_j=15$.
\[
\begin{aligned}
\mathbf B_{23,0}={}\bigl( & 5\,060\,439\,269\,772\,109,\; 3\,357\,663\,479\,749\,356,\; -942\,941\,749\,796\,635,\; -236\,032\,571\,772\,692,\\
& 298\,257\,870\,198\,364,\; 185\,703\,990\,276\,974,\; 5\,404\,178\,001\,568,\; -37\,631\,473\,945\,072,\\
& -13\,470\,586\,012\,421,\; 7\,492\,845\,515\,210,\; 7\,180\,831\,299\,687,\; 150\,154\,248\,401,\\
& -2\,591\,400\,121\,843,\; -463\,162\,470\,214,\; 1\,128\,149\,769\,679,\; 750\,625\,720\,136\bigr)
\end{aligned}
\]
\[
\begin{aligned}
\mathbf B_{23,1}={}\bigl( & -1\,219\,886\,737\,066\,740,\; 18\,545\,465\,181\,485,\; 97\,273\,459\,982\,067,\; -19\,028\,547\,998\,250,\\
& -10\,545\,251\,182\,750,\; 5\,644\,573\,907\,055,\; 4\,241\,045\,448\,474,\; -110\,108\,508\,849,\\
& -1\,196\,884\,220\,973,\; -305\,536\,995\,065,\; 295\,465\,381\,970,\; 207\,459\,688\,043,\\
& -60\,855\,622\,195,\; -90\,184\,339\,058,\; -4\,617\,489\,955,\; 25\,170\,082\,944\bigr)
\end{aligned}
\]
\[
\mathbf G_{23}=(-271\,581\,166\,090\,707,\quad 18\,363\,313\,613\,461),\qquad
\mathbf E_{23}=(336\,087\,300\,661\,143,\quad 3\,163\,099\,199\,270).
\]
\end{samepage}
\normalsize
\par\penalty0\smallskip

\begin{samepage}
\footnotesize
\setlength{\abovedisplayskip}{5pt}
\setlength{\belowdisplayskip}{5pt}
\noindent $j=24$, $I_j=[\frac{89}{20},\frac{49}{10}]$, $N_j=15$.
\[
\begin{aligned}
\mathbf B_{24,0}={}\bigl( & 5\,130\,685\,965\,203\,766,\; 2\,971\,043\,061\,442\,524,\; -822\,059\,365\,499\,906,\; -229\,964\,393\,624\,212,\\
& 280\,445\,839\,811\,330,\; 183\,822\,783\,530\,365,\; 8\,164\,951\,478\,024,\; -36\,865\,265\,887\,389,\\
& -13\,959\,543\,266\,593,\; 7\,227\,223\,046\,974,\; 7\,260\,963\,331\,357,\; 236\,787\,065\,221,\\
& -2\,602\,995\,213\,562,\; -489\,955\,653\,527,\; 1\,126\,825\,664\,284,\; 757\,987\,876\,654\bigr)
\end{aligned}
\]
\[
\begin{aligned}
\mathbf B_{24,1}={}\bigl( & -1\,872\,709\,627\,489\,539,\; -49\,948\,660\,554\,502,\; 137\,861\,269\,809\,993,\; -11\,348\,706\,625\,910,\\
& -22\,262\,648\,167\,562,\; -1\,323\,056\,018\,928,\; 3\,778\,925\,429\,647,\; 1\,078\,299\,163\,748,\\
& -641\,065\,949\,440,\; -440\,375\,587\,316,\; 64\,989\,964\,468,\; 155\,456\,932\,055,\\
& 9\,723\,830\,328,\; -55\,381\,155\,754,\; -24\,943\,255\,573,\; 3\,451\,731\,015\bigr)
\end{aligned}
\]
\[
\mathbf G_{24}=(-254\,501\,140\,469\,729,\quad 30\,339\,797\,970\,915),\qquad
\mathbf E_{24}=(293\,185\,093\,326\,804,\quad 4\,882\,128\,719\,314).
\]
\end{samepage}
\normalsize
\par\penalty0\smallskip

\begin{samepage}
\footnotesize
\setlength{\abovedisplayskip}{5pt}
\setlength{\belowdisplayskip}{5pt}
\noindent $j=25$, $I_j=[\frac{97}{20},\frac{11}{2}]$, $N_j=15$.
\[
\begin{aligned}
\mathbf B_{25,0}={}\bigl( & 5\,290\,996\,271\,473\,685,\; 2\,560\,749\,608\,025\,492,\; -713\,095\,035\,559\,578,\; -215\,165\,272\,734\,915,\\
& 260\,159\,173\,165\,442,\; 178\,554\,479\,468\,887,\; 10\,651\,871\,464\,416,\; -35\,436\,332\,874\,431,\\
& -14\,229\,280\,775\,599,\; 6\,826\,726\,327\,791,\; 7\,227\,488\,898\,073,\; 325\,208\,799\,936,\\
& -2\,577\,513\,297\,234,\; -507\,755\,180\,965,\; 1\,113\,573\,372\,741,\; 755\,094\,238\,662\bigr)
\end{aligned}
\]
\[
\begin{aligned}
\mathbf B_{25,1}={}\bigl( & -2\,667\,364\,138\,401\,118,\; -170\,231\,604\,365\,306,\; 185\,319\,976\,395\,995,\; 4\,390\,930\,929\,631,\\
& -40\,287\,183\,319\,011,\; -15\,227\,453\,668\,510,\; 1\,866\,169\,969\,927,\; 3\,354\,208\,115\,987,\\
& 675\,049\,803\,600,\; -678\,926\,309\,032,\; -461\,725\,606\,031,\; 24\,344\,641\,655,\\
& 169\,322\,682\,504,\; 27\,309\,947\,638,\; -69\,043\,540\,281,\; -45\,598\,634\,990\bigr)
\end{aligned}
\]
\[
\mathbf G_{25}=(-237\,857\,436\,433\,386,\quad 47\,125\,326\,222\,266),\qquad
\mathbf E_{25}=(244\,393\,696\,918\,888,\quad 6\,887\,317\,044\,188).
\]
\end{samepage}
\normalsize
\par\penalty0\smallskip

\begin{samepage}
\footnotesize
\setlength{\abovedisplayskip}{5pt}
\setlength{\belowdisplayskip}{5pt}
\noindent $j=26$, $I_j=[\frac{109}{20},\frac{13}{2}]$, $N_j=15$.
\[
\begin{aligned}
\mathbf B_{26,0}={}\bigl( & 6\,157\,871\,434\,992\,430,\; 2\,249\,690\,166\,867\,241,\; -677\,824\,437\,380\,257,\; -198\,720\,001\,457\,198,\\
& 251\,589\,034\,915\,219,\; 174\,945\,331\,060\,863,\; 11\,284\,828\,265\,298,\; -34\,865\,329\,627\,295,\\
& -14\,249\,334\,265\,922,\; 6\,777\,878\,535\,887,\; 7\,241\,945\,332\,585,\; 319\,025\,722\,372,\\
& -2\,591\,338\,885\,263,\; -497\,484\,777\,131,\; 1\,130\,285\,282\,018,\; 760\,593\,394\,693\bigr)
\end{aligned}
\]
\[
\begin{aligned}
\mathbf B_{26,1}={}\bigl( & -4\,219\,702\,192\,804\,876,\; -466\,394\,271\,355\,822,\; 300\,045\,382\,023\,411,\; 35\,478\,482\,021\,993,\\
& -82\,690\,285\,312\,408,\; -45\,579\,129\,792\,902,\; -812\,067\,298\,597,\; 9\,133\,447\,721\,955,\\
& 3\,317\,343\,413\,732,\; -1\,688\,927\,845\,657,\; -1\,725\,031\,045\,556,\; -99\,424\,981\,304,\\
& 605\,971\,680\,719,\; 142\,276\,019\,426,\; -246\,478\,035\,017,\; -177\,580\,629\,605\bigr)
\end{aligned}
\]
\[
\mathbf G_{26}=(-235\,957\,289\,131\,428,\quad 83\,824\,752\,295\,395),\qquad
\mathbf E_{26}=(188\,662\,530\,504\,966,\quad 9\,726\,563\,249\,672).
\]
\end{samepage}
\normalsize
\par\penalty0\smallskip

\begin{samepage}
\footnotesize
\setlength{\abovedisplayskip}{5pt}
\setlength{\belowdisplayskip}{5pt}
\noindent $j=27$, $I_j=[\frac{32}{5},\frac{15}{2}]$, $N_j=11$.
\[
\begin{aligned}
\mathbf B_{27,0}={}\bigl( & 8\,773\,445\,318\,123\,258,\; 2\,342\,748\,743\,776\,263,\; -802\,626\,795\,255\,034,\; -211\,994\,405\,865\,063,\\
& 292\,343\,696\,557\,315,\; 207\,753\,913\,363\,688,\; 21\,590\,791\,379\,844,\; -37\,002\,870\,273\,291,\\
& -16\,483\,709\,127\,584,\; 6\,688\,899\,963\,091,\; 9\,124\,516\,161\,673,\; 3\,414\,005\,781\,841\bigr)
\end{aligned}
\]
\[
\begin{aligned}
\mathbf B_{27,1}={}\bigl( & -4\,170\,382\,713\,450\,770,\; -601\,784\,488\,331\,548,\; 309\,449\,913\,593\,582,\; 51\,164\,399\,999\,208,\\
& -96\,297\,827\,677\,257,\; -60\,433\,389\,920\,746,\; -4\,555\,514\,662\,241,\; 10\,998\,568\,950\,988,\\
& 4\,493\,310\,041\,735,\; -2\,012\,578\,309\,035,\; -2\,572\,714\,474\,838,\; -943\,886\,621\,232\bigr)
\end{aligned}
\]
\[
\mathbf G_{27}=(-288\,588\,164\,028\,711,\quad 87\,897\,626\,377\,879),\qquad
\mathbf E_{27}=(136\,615\,742\,628\,300,\quad 8\,626\,295\,753\,503).
\]
\end{samepage}
\normalsize
\par\penalty0\smallskip

\begin{samepage}
\footnotesize
\setlength{\abovedisplayskip}{5pt}
\setlength{\belowdisplayskip}{5pt}
\noindent $j=28$, $I_j=[\frac{37}{5},\frac{93}{10}]$, $N_j=11$.
\[
\begin{aligned}
\mathbf B_{28,0}={}\bigl( & 53\,244\,940\,789\,910\,616,\; 10\,477\,569\,929\,210\,052,\; -4\,108\,969\,669\,903\,220,\; -964\,675\,657\,349\,664,\\
& 1\,470\,841\,342\,230\,797,\; 1\,055\,964\,288\,675\,652,\; 87\,132\,840\,154\,234,\; -216\,588\,047\,054\,789,\\
& -67\,196\,322\,185\,460,\; 98\,912\,868\,146\,038,\; 101\,462\,795\,540\,758,\; 57\,199\,529\,656\,218\bigr)
\end{aligned}
\]
\[
\begin{aligned}
\mathbf B_{28,1}={}\bigl( & 6\,514\,983\,813\,432\,037,\; 2\,340\,887\,920\,819\,844,\; -593\,473\,531\,433\,028,\; -242\,778\,653\,380\,990,\\
& 281\,234\,898\,989\,803,\; 238\,315\,297\,826\,084,\; -7\,158\,261\,215\,357,\; -78\,262\,284\,856\,021,\\
& 47\,998\,342\,064,\; 77\,092\,550\,102\,510,\; 67\,049\,296\,311\,739,\; 43\,371\,422\,928\,759\bigr)
\end{aligned}
\]
\[
\mathbf G_{28}=(-1\,431\,642\,595\,692\,508,\quad -249\,458\,865\,946\,376),\qquad
\mathbf E_{28}=(-63\,334\,303\,029\,205,\quad -2\,056\,920\,616\,513).
\]
\end{samepage}
\normalsize
\par\penalty0\smallskip

\begin{samepage}
\footnotesize
\setlength{\abovedisplayskip}{5pt}
\setlength{\belowdisplayskip}{5pt}
\noindent $j=29$, $I_j=[9,11]$, $N_j=11$.
\[
\begin{aligned}
\mathbf B_{29,0}={}\bigl( & 136\,915\,288\,533\,818\,168,\; 27\,564\,563\,721\,684\,760,\; -11\,007\,552\,888\,738\,462,\; -2\,490\,208\,161\,632\,625,\\
& 3\,818\,771\,072\,503\,582,\; 2\,658\,442\,790\,222\,726,\; 292\,383\,720\,313\,469,\; -461\,175\,732\,343\,320,\\
& -215\,430\,108\,566\,864,\; 77\,834\,035\,767\,509,\; 114\,400\,113\,934\,159,\; 44\,472\,208\,190\,722\bigr)
\end{aligned}
\]
\[
\begin{aligned}
\mathbf B_{29,1}={}\bigl( & -9\,749\,747\,086\,988\,028,\; -782\,474\,508\,860\,664,\; 673\,412\,100\,859\,312,\; 50\,524\,327\,347\,845,\\
& -176\,582\,016\,792\,970,\; -79\,261\,947\,211\,366,\; 11\,407\,717\,352\,225,\; 21\,099\,100\,799\,879,\\
& 2\,507\,658\,597\,015,\; -6\,200\,770\,632\,917,\; -3\,985\,970\,934\,580,\; -697\,317\,513\,335\bigr)
\end{aligned}
\]
\[
\mathbf G_{29}=(-3\,999\,999\,999\,999\,926,\quad 0),\qquad
\mathbf E_{29}=(-333\,515\,814\,144\,031,\quad 58\,868\,288\,538\,577).
\]
\end{samepage}
\normalsize
\par\penalty0\smallskip

\bigskip

\noindent
{\bf Acknowledgements.}

 The research of S. Luo is partially supported by the National Natural Science Foundation of China (NSFC) under Grant Nos. 12261045 and 12001253, and by the Jiangxi Jieqing Fund under Grant No. 20242BAB23001.
 The research of J. Wei is partially supported by the General Research Fund (GRF) of Hong Kong "New frontiers in singularity formation of nonlinear partial differential equations".
ChatGPT was used to assist with mathematical exposition, review of
proof arguments and citations, and preparation of the computational
verification materials (including all the appendices).

\renewcommand{\bibliofont}{\footnotesize\setlength{\baselineskip}{10pt}}

\end{document}